\documentclass[10pt,a4paper]{amsart}
\usepackage{amsmath, graphicx, subfigure, epsfig,amssymb}
\usepackage{xcolor}
\numberwithin{equation}{section}

\newcommand{\e}{\epsilon}

\newcommand{\de}{\delta}
\newcommand{\br}{\mathbb{R}}
\newcommand{\N}{\mathbb{N}}
\newcommand{\ik}{\varphi}
\newcommand{\pa}{\partial}
\newcommand{\bt}{\beta}
\newcommand{\al}{\alpha}
\newcommand{\la}{\lambda}
\newcommand{\om}{\omega}
\newcommand{\coi}{C_0^{\infty}}
\newcommand{\ioi}{\int_0^{\infty}}

\newcommand{\be}{\begin{equation}}
\newcommand{\ee}{\end{equation}}

\newcommand{\fs}{\tilde\upsilon}

\newcommand{\os}{{(1)}}
\newcommand{\tp}{{(2)}}

\newcommand{\bma}{\begin{pmatrix}}
\newcommand{\ema}{\end{pmatrix}}

\newcommand{\CA}{\mathcal{A}}
\newcommand{\CW}{\mathcal{W}}

\newcommand{\us}{\mathcal U}
\newcommand{\vs}{\mathcal V}

\newcommand{\s}{\mathcal S}
\newcommand{\T}{\mathcal T}
\newcommand{\R}{\mathcal R}
\newcommand{\CF}{\mathcal F}

\newcommand{\Yo}{Y^{(1)}}
\newcommand{\yo}{y^{(1)}}
\newcommand{\yt}{y^{(2)}}
\newcommand{\ytt}{\check y^{(2)}}
\newcommand{\xo}{x^{(1)}}
\newcommand{\Xo}{X^{(1)}}
\newcommand{\xt}{x^{(2)}}
\newcommand{\vy}{\hat y}
\newcommand{\yor}{\tilde y}
\newcommand{\vor}{\tilde\vs}
\newcommand{\Por}{\tilde\Phi}

\newcommand{\gao}{s_0}
\newcommand{\CB}{\mathcal B}
\newcommand{\CBa}{\mathcal{B}_{1d}}
\newcommand{\CE}{\mathcal E}

\newcommand{\bp}{\Theta}
\newcommand{\dsf}{\Delta\text{II}}
\newcommand{\CTB}{\text{CTB}}
\newcommand{\DTB}{\text{DTB}}

\newcommand{\MPsi}{\Upsilon}
\newcommand{\Mpsi}{\phi}

\theoremstyle{theorem}
\newtheorem{theorem}{Theorem}

\newtheorem{lemma}{Lemma}

\newtheorem{definition}{Definition}

\begin{document}

\title[Analysis of resolution]{Resolution analysis of inverting the generalized $N$-dimensional Radon transform in $\br^n$ from discrete data}
\author[A Katsevich]{Alexander Katsevich$^1$}
\thanks{$^1$This work was supported in part by NSF grant DMS-1906361. Department of Mathematics, University of Central Florida, Orlando, FL 32816 (Alexander.Katsevich@ucf.edu). }

\begin{abstract} 
Let $\R$ denote the generalized Radon transform (GRT), which integrates over a family of $N$-dimensional smooth submanifolds $\s_{\yor}\subset\us$, $1\le N\le n-1$, where an open set $\us\subset\br^n$ is the image domain. The submanifolds are parametrized by points $\yor\subset\vor$, where an open set $\vor\subset\br^n$ is the data domain. The continuous data are $g=\R f$, and the reconstruction is $\check f=\R^*\CB g$. Here $\R^*$ is a weighted adjoint of $\R$, and $\CB$ is a pseudo-differential operator. We assume that $f$ is a conormal distribution, $\text{supp}(f)\subset\us$, and its singular support is a smooth hypersurface $\s\subset\us$. Discrete data consists of the values of $g$ on a lattice $\yor^j$ with the step size $O(\e)$. Let $\check f_\e=\R^*\CB g_\e$ denote the reconstruction obtained by applying the inversion formula to an interpolated discrete data $g_\e(\yor)$. Pick a generic pair $(x_0,\yor_0)$, where $x_0\in\s$, and $\s_{\yor_0}$ is tangent to $\s$ at $x_0$. The main result of the paper is the computation of the limit 
$$
f_0(\check x):=\lim_{\e\to0}\e^\kappa \check f_\e(x_0+\e\check x).
$$
Here $\kappa\ge0$ is selected based on the strength of the reconstructed singularity, and $\check x$ is confined to a bounded set. The limiting function $f_0(\check x)$, which we call the discrete transition behavior, allows computing the resolution of reconstruction.  
\end{abstract}
\maketitle

\section{Introduction}\label{sec_intro}

Analysis of resolution of tomographic reconstruction of a function $f$ from its discrete Radon transform data is a practically important problem. Usually, it is solved in the setting of the sampling theory applied to the classical Radon transform that integrates over hyperplanes. The key assumption in this approach is that $f$ is essentially bandlimited \cite{nat93, pal95, far04}. An extension of this theory allows consideration of more general Radon transforms and allows $f$ to have at most semiclassical singularities \cite{stef20}. 

Frequently, one would like to know how accurately and with what resolution the {\it classical} singularities of $f$ (e.g., a jump discontinuity across a smooth surface $\s=\text{sing\hspace{1.0pt}supp}(f)$) are reconstructed. Let $\check f$ denote the reconstruction from continuous data, and $\check f_\e$ denote the reconstruction from discrete data, where $\e$ represents the data sampling rate. In the latter case, interpolated discrete data are substituted into the ``continuous'' inversion formula. In \cite{kat_2017, kat19a, kat20b, kat20a} the author initiated the analysis of reconstruction by focusing specifically on the behavior of $\check f_\e$ near $\s$. One of the main results of these papers is the computation of the limit 
\be\label{tr-beh}
\DTB(\check x):=f_0(\check x):=\lim_{\e\to0}\e^\kappa \check f_\e(x_0+\e\check x).
\ee
Here $x_0\in\s$ is generic (see a more precise definition below), $\kappa\ge0$ is selected based on the strength of the singularity of $\check f$ at $x_0$, and $\check x$ is confined to a bounded set. It is important to emphasize that both the size of the neighborhood around $x_0$ and the data sampling rate go to zero simultaneously in \eqref{tr-beh}. The limiting function $f_0(\check x)$, which we call the discrete transition behavior (or DTB for short), contains complete information about the resolution of reconstruction.  

The results obtained to date can be summarized as follows. Even though we study reconstruction from discrete data, the classification of the cases is based on their continuous analogues. In \cite{kat_2017} we find $f_0(\check x)$ for the Radon transform in $\br^2$ in two cases: $f$ is static and $f$ changes during the scan (dynamic tomography). In the static case the reconstruction formula is exact (i.e., $\check f=f$), and in the dynamic case the reconstruction formula is quasi-exact (i.e., $\check f-f$ is smoother than $f$). In \cite{kat19a} we find $f_0(\check x)$ for the classical Radon transform (CRT) in $\br^3$ assuming the reconstruction is exact and $f$ has jumps. In \cite{kat20a} we consider a similar setting as in \cite{kat19a}, i.e. $f$ has jumps and reconstruction is quasi-exact, but consider a wide family of generalized Radon transforms (GRT) in $\br^3$. Finally, in \cite{kat20b}, the data still comes from the classical Radon transform, but the dimension is increased to $\br^n$, the reconstruction operators are more general, and $f$ may have singularities other than jumps. See Table~\ref{tbl:cases} for a summary of the cases.

\begin{table}
\begin{center}
\begin{tabular}{ |c|c|c|c|c| } 
 \hline
  & RT type & dimension & type of inversion & singularity of $f$ \\ 
\hline
  \cite{kat_2017}  & CRT \& GRT & 2 & exact/quasi-exact & jumps \\ 
 \cite{kat19a} & CRT  & 3 & exact & jumps \\ 
 \cite{kat20a} & GRT & 3 & quasi-exact & jumps \\ 
 \cite{kat20b} & CRT & $n$ & more general & more general \\ 
 \hline
\end{tabular}
\end{center}
\caption{Summary of the cases considered prior to this paper.}
\label{tbl:cases}
\end{table}

Let $\R$ denote the GRT, which integrates over a family of $N$-dimensional smooth submanifolds $\s_{\yor}\subset\us\subset\br^n$, $1\le N\le n-1$. An open set $\us$ represents the image domain. The submanifolds $\s_{\yor}$ are parametrized by points $\yor\subset\vor$, where an open set $\vor\subset\br^n$ is the data domain. Reconstruction from continuous data $g=\R f$ is obtained by $\check f=\R^*\CB g$. Here $\R^*$ is a weighted adjoint of $\R$, which integrates over submanifolds $\tilde\T_x:=\{\yor\in\vor:x\in\s_{\yor}\}$, and $\CB$ is a fairly general pseudo-differential operator ($\Psi$DO). In fact, $g$ does not even have to be the GRT of some $f$. All we need is that $g$ be a sufficiently regular conormal distribution associated with a smooth hypersurface $\Gamma\subset\vs$.

If $g=\R f$, we allow $f$ to have a fairly general singularity across a smooth hypersurface. The choice of $\CB$ determines whether the reconstruction is quasi-exact (i.e., $\check f-f$ is smoother than $f$), preserves the order of singularities of $f$ ($\check f$ and $f$ are in the same Sobolev space), or is singularity-enhancing ($\check f$ is more singular than  $f$). A common example of the latter is Lambda (also known as local) tomography \cite{rk, fbh01}.  

Thus, the setting considered in this paper includes all the cases considered previously \cite{kat_2017, kat19a, kat20a, kat20b}, but is substantially more general than before. In particular, in previous work we always had $N=n-1$. Now, $N$ can be any integer $1\le N\le n-1$. This includes the practically most important  case of cone beam CT: $n=3$ and $N=1$, on which the overwhelming majority of all medical, industrial, and security CT scans are based (see e.g. \cite{kal11, orh20} and references therein). The main result of this paper is the derivation of the DTB \eqref{tr-beh} under these general conditions. We also show that the DTB equals to the convolution of the continuous transition behavior (CTB) with the suitably scaled classical Radon transform of the interpolation kernel. Loosely speaking, the CTB is the continuous analogue of the DTB: 
\be\label{CTB-beh}
\CTB(\check x):=\lim_{\e\to0}\e^\kappa (R^*\CB g)(x_0+\e\check x).
\ee

The operator $\R^*$ can be viewed as a Fourier Integral Operator (FIO), which is associated to a phase function linear in the frequency variables (see Section 2.4 in \cite{hor71} and Section 1.3 in \cite{gs02}):
\be
(\R^*g)(x)=\frac1{(2\pi)^N}\int_{\br^N}\int_{\vor} e^{i\mu\cdot\Psi(x,\yor)}w(x,\yor)g(\yor)d\yor d\mu,
\ee
where $w\in \coi(\us\times\vor)$, and $\Psi\in C^\infty(\us\times\vor)$ is any $\br^{n-N}$ valued function that satisfies some nondegeneracy conditions. Any such $\Psi$ determines a pair $\R$, $\R^*$ by setting $\s_{\yor}=\{x\in\us:\,x=\Por(x,\yor)\}$, $\T_x=\{y\in\vor:\,x=\Por(x,\yor)\}$, and selecting integration weights. As is easily seen, any FIO $F$ with the same phase can be represented in the form $F=\R^*\CB$ for some $\CB$ (at least, microlocally where $\R^*$ is elliptic). Indeed, we can just take $\CB=(\R^*)^{-1}F$, where $(\R^*)^{-1}$ is an approximate inverse of $\R^*$ (modulo regularizing operators). We assume here that an appropriate cut-off is introduced in $(\R^*)^{-1}$, so the composition is well-defined. Then $\CB$ is a $\Psi$DO, and $F-\R^*\CB$ is regularizing. Thus, the reconstruction algorithm is the application of an FIO $F$ with a phase function, which is linear in the frequency variables, to discrete data $g(\yor_j)$. 

Various methods for applying FIOs to discrete data have been proposed, see e.g. \cite{cdy07, cdy09, ahw12, yang15} and references therein. This appears to be the {\it first analysis of resolution} of the reconstructed image $Fg$ for fairly general classes of FIOs $F$ and distributions $g$. Some results along this direction are in \cite{kat20a}. Analyses of such sort are especially important, because most frequently FIO-based reconstruction is designed to accurately recover only the singularities of the unknown original object $f$. Reconstruction of the smooth part of $f$ is usually not accurate even if the data were ideal (i.e., known everywhere). See, for example, Remark 1 in \cite{salo20}. Our approach, which we call local resolution analysis, is well suited to the analysis of such linear recosntruction algorithms because the analysis is localized to an immediate neighborhood of the singularities of $f$. 

A powerful theory, which can be used to study resolution, is semiclassical analysis \cite{zwor12, stef20}. Nevertheless, the development here is novel and not anticipated by that theory. A common thread through our work is that the well-behaved DTB (i.e., the limit in \eqref{tr-beh}) is guaranteed to exist only if a pair $(x_0,\yor_0)\in\us\times\vor$ is generic. Here $x_0\in\s$, and $y_0\in\Gamma:=\text{sing\hspace{1.0pt}supp}(g)$ is the data point from which the singularity of $f$ at $x_0$ is visible, i.e. $\s_{\yor_0}$ is tangent to $\s$ at $x_0$. Roughly, the pair is generic (or, locally generic, to be precise) if in a small neighborhood of $\yor_0$ the sampling lattice $\yor^j$ is in general position relative to a local patch of $\Gamma$ containing $\yor_0$. Equivalently, we can define what it means for $(x_0,\xi_0)\in WF(\check f)$ to be generic, because, generally, $(x_0,\yor_0)$ can be found from $(x_0,\xi_0)$. The property of a pair to be generic is closely related with the uniform distribution theory \cite{KN_06}. 

If $(x_0,\yor_0)$ is not generic, the DTB may be different from the generic one predicted by our theory, and certain non-local artifacts that depend on the shape of $\s$ can appear as well (see e.g. \cite{kat20b}) even if $\R^*\R$ is a $\Psi$DO. These are novel phenomena, which appear to be outside the scope of semiclassical analysis in its present form \cite{zwor12, stef20} (the artifacts described in \cite{stef20} and \cite{kat20b} are different). This shows also that the case of discrete data is more complicated than when the data are continuous, because in the latter case $\text{WF}(\check f)\subset \text{WF}(f)$ whenever $\R^*\R$ is a $\Psi$DO. 

The paper is organized as follows. In Section~\ref{sec:a-prelims} we introduce the generalized Radon transform and its adjoint, the sampling matrix, the sampling lattice $\yor^j$, and fix a pair $(x_0,\yor_0)\in\us\times\vor$ such that $\s_{\yor_0}$ is tangent to $\s$ at $x_0$. In Section~\ref{sec:prelims} we select convenient coordinates both in the data and image domains, state the main geometric assumptions about the Radon transform $\R$ and the shape of $\s$, and state the definition of a generic pair. In Section~\ref{sec:tang} we show that $\T_\s$ and $\T_{x_0}$ are tangent at $y_0$. Here $\T_\s$ is the surface consisting of all $y\in \vs$ such that $\s_y$ is tangent to $\s$. All calculations are done in the new $y$-coordinates, so we drop the tildas in $\vor$, $\yor$, etc. We also look at the distance between $\T_{x_0}$ and $\T_\s$ for points $y$ that satisfy $|y-y_0|=O(\e^{1/2})$, and obtain a local equation for $\T_{x_\e}$, where $|x_\e-x_0|=O(\e)$. Next, in Section~\ref{GRT-behavior} we describe the assumptions about $f$ and obtain convenient formulas for the behavior of $\R f$ near its singular support. Essentially, both $f$ and $\R f$ are conormal distributions (see Section 18.2 in \cite{hor3}) associated with smooth codimension one surfaces $\s$ and $\T_\s$, respectively. The fact that $\R f$ is conormal follows from the calculus of FIOs, see e.g. Section VIII.5 in \cite{trev2}. We present the necessary calculations here, because they are short, make the paper self-contained, and these calculations are used elsewhere in the paper.

In Section~\ref{sec:assns} we formulate the main assumptions that the $\Psi$DO $\CB$, interpolation kernel $\ik$, and the distribution $g$ satisfy. Our goal is to study the DTB determined by $\R^*\CB$ acting on an interpolated version of the discrete data $g(\yor^j)$. The data does not have to be of the form $g=\R f$ for some $f$. Basically, we require that $g$ be a conormal distribution associated with a smooth hypersurface $\Gamma\subset\vs$. If $g=\R f$, then $\Gamma=\T_\s$. We also break up the reconstruction formula into two parts. In the first part, the integration over $\T_{x_\e}$ (which is performed when $\R^*$ is applied) is restricted to a domain centered at $y_0$ and of size $A\e^{1/2}$ for some fixed $A\gg1$. In the second part, integration is over $\T_{x_\e}$ outside the $A\e^{1/2}$ size neighborhood centered at $y_0$. The first and second parts of the reconstruction are denoted $f_\e^{\os}$ and $f_\e^{\tp}$, respectively.

Let $g_\e$ be the interpolated data. In Section~\ref{sec:gprops} we obtain various bounds on $g$, $g_\e$, $g-g_\e$, and their derivatives. In Section~\ref{sec:fplt} we obtain the limit of the first part of the reconstruction $f_\e^{\os}$ as $\e\to0$ assuming that the symbols of both $\CB$ and $g$ contain only their top order terms. Computation of this limit uses that $(x_0,\yor_0)$ is generic and is based on the uniform distribution theory, see \eqref{key-int}--\eqref{lim}. In Section~\ref{sec:fplots} we obtain the limit of $f_\e^{\os}$ as $\e\to0$ assuming that at least one of the symbols, either that of $\CB$ or $g$, does not contain its top order term. In Section~\ref{sec:spdtb} we show that the second part of the reconstruction $f_\e^{\tp}$ does not contribute to the DTB in the limit as $A\to\infty$. In Section~\ref{sec:compdtb} we compute the DTB by considering the limit as $A\to\infty$. We also state our main results there. The CTB is computed in Section~\ref{sec:compctb}, and we establish that the DTB is the convolution of the CTB and the properly scaled classical Radon transform of the interpolation kernel. Finally, proofs of various lemmas are collected in the appendices.

\section{Preliminaries}\label{sec:a-prelims}

Let $\Por(t,\yor)\in C^\infty(\br^N\times\vor)$ be a defining function for the GRT $\R$. Here $t\in\br^N$ is an auxiliary variable that parametrizes smooth manifolds $\s_{\yor}:=\{x\in\us:x=\Por(t,\yor),t\in\br^N\}$ over which $\R$ integrates, an open set $\us\subset\br^n$ is the image domain, $\yor\in\vor$ is the data domain variable, and an open set $\vor\subset\br^n$ is the data domain. The corresponding GRT is given by
\be\label{lead-t-alt-1}
\R f(\yor)= \int_{\s_{\yor}} f(x)b(x,\yor) dx,
\ee
where $dx$ is the volume form on $\s_{\yor}$ and $b\in \coi(\us\times\vor)$. We assume that $f$ is compactly supported, $\text{supp}(f)\subset\us$, and $f$ is sufficiently smooth, so that $\R f(\yor)$ is a continuous function. Exact reconstruction is computed by
\be\label{recon_cont_or}
\check f(x)=(\R^*\CB g)(x)=\int_{\tilde \T_x}(\CB g)(\yor) w(x,\yor)d\yor,\ g=\R f,
\ee
where $w\in \coi(\us\times\vor)$, $d\yor$ is the volume form on $\tilde\T_x:=\{\yor\in\vor:x\in\s_{\yor}\}$, $\R^*$ is a weighted adjoint of $\R$, and $\CB$ is a fairly arbitrary pseudo-differential operator ($\Psi$DO). The reconstruction formula in \eqref{recon_cont_or} is of the Filtered-Backprojection type. Application of $\CB$ is the filtering step, and integration with respect to $\yor$ (i.e., the application of $\R^*$) is the backprojection step.  By reconstruction here we mean any function (or, distribution) $\check f$ that is reconstructed from the data using \eqref{recon_cont_or}. The reconstruction is intended to recover the visible wave-front set of $f$, but the strength of the singularities of $\check f$ and $f$ need not match. 

Let $D$ be a data sampling matrix, $\det D=1$. Discrete data $g(\yor^j)$ are given on a lattice
\be\label{data-pts-or}
\yor^j=\e Dj,\ j\in\mathbb Z^n.
\ee
Reconstruction from discrete data is given by the same formula \eqref{recon_cont_or}, where we replace $g$ with its interpolated version $g_\e$.

We assume that $\s:=\text{sing\hspace{1.0pt}supp}(f)$ is a smooth hypersurface. Pick some $x_0\in\s$. This point is fixed throughout the paper. Our goal is to study the function reconstructed from discrete data in a neighborhood of $x_0$. All our results are local, so we assume that $\us$ is a sufficiently small neighborhood of $x_0$. Let $\yor_0\in\vor$ be such that $\s_{\yor_0}$ is tangent to $\s$ at $x_0$. Only a small neighborhood of $\yor_0$ is relevant for the recovery of the singularity of $f$ at $x_0$. Hence we assume that $\vor$ is a sufficiently small neighborhood of $\yor_0$.

\section{Selecting coordinates, geometric assumptions}\label{sec:prelims}

Later we will select a convenient $y$ coordinate system. Since data points in \eqref{data-pts-or} are given in the original coordinates, we have to keep track of both the original and new coordinates.  Points in the original and new $y$ coordinates are denoted $\yor$ and $y$, respectively. Data domains in the original and new $y$ coordinates are denoted $\vor$ and $\vs$, respectively. Suppose that $\yor=Uy+\yor_0$ and $\vor=U\vs+\yor_0$, where $U$ is an orthogonal matrix $U:\br^n\to\br^n$ to be selected below. 

In what follows we will be using mostly the new $y$ coordinates, so we modify the defining function appropriately:
\be\label{newdeffn}
\Phi(t,y):=\Por(t,\yor(y))=\Por(t,Uy+\yor_0).
\ee
The new $y$ coordinates are selected so that
\be\label{y-coords}
y=\bma y_1\\ y_\perp\ema,\ y_1\in\br,\ y_\perp\in\br^{n-1},\ y_0=\bma 0\\ 0\ema, (\Psi\circ\Phi)_{y_1}=1,\  (\Psi\circ\Phi)_{y_\perp}=0.
\ee
Here $\Psi(x)=0$ is an equation of $\s$, and $d\Psi(x)\not=0$, $x\in\us$. Multiplying $\Psi(y)$ by a constant, we can make sure that $(\Psi\circ\Phi)_{y_1}=1$. For convenience, here and in the rest of the paper we frequently drop the arguments of $\Psi$, $\Phi$, and similar functions whenever they are $x_0$ and $(t_0=0,y_0=0)$, as appropriate. 

Our convention is that a variable in the subscript of a function denotes the partial derivative (or gradient) of the function with respect to that variable (or a group of variables), e.g. $\Phi_{y_1}=\pa_{y_1}\Phi$. 

Suppose $x$ and $t$ coordinates are selected so that 
\be\label{xcoords}\begin{split}
&x=\bma\xo\\ \xt\ema,\ \xo\in\br^{n-N},\ \xt\in\br^N,\ x_0=\bma 0\\0\ema=\Phi(t=0,y_0=0),\\
&d\Psi=(|d\Psi|,0,\dots,0),\ \Phi_t^{(1)}=0,\ \det \Phi_t^{(2)}\not=0.
\end{split}
\ee
We also denote $x_\perp:=(x_2,\dots,x_n)^T$. 
The notation $\Phi_{*}^{(j)}$, $j=1,2$, stands for the derivative of the $j$-th group of coordinates of $x=\Phi(t,y)$ (either along $\xo$ or along $\xt$) with respect to $*=t$ or $y$. Thus, $\s_y\subset\us$ is a smooth $N$-dimensional embedded submanifold for any $y\in\vs$.

Let $\pi$ be the orthogonal projection onto the $(\xo_1,\xt)$ coordinates: $x\to \bma \xo_1\\ \xt \ema$. Here $\xo_1$ (which is the same as $x_1$) denotes the first component of the first group of coordinates $\xo$. We have
\be\label{jacob-sm}\begin{split}
(\Psi\circ\Phi)_{tt}=&(\Psi''\Phi_{t_k}\Phi_{t_j}+\Psi'\Phi_{t_kt_j})_{1\le k,j\le N}
=|d\Psi|\bma\Phi_t^{(2)}\ema^T\dsf\, \Phi_t^{(2)},\\
\dsf:=&\text{II}_{\hat\s_{y_0}}(x_0)-\text{II}_{\hat\s}(x_0).
\end{split}
\ee
Here $\hat\s=\pi\s$, $\hat\s_{y_0}=\pi\s_{y_0}$, and $\text{II}_{\hat\s}(x)$ is the matrix of the second fundamental form of $\hat\s$ at $\pi x\in \hat\s$ written in the coordinates $\bma \xo_1\\ \xt\ema$. 

Let $\T_\s$ be the set, which consists of all $y\in\vs$ such that  $\s_y$ is tangent to $\s$. Similarly, $\T_x$ denotes the set, which consists of all $y\in\vs$ such that  $\s_y$ contains $x$. To find $\T_\s$, we solve the equations 
\be\label{ts-eqns}
\Psi(\Phi(t,y))=0,\ d\Psi(\Phi(t,y))\Phi_t(t,y)=(\Psi\circ\Phi)_t(t,y)=0.
\ee
The Jacobian matrix is
\be\label{jacob-1}
\begin{pmatrix}
(\Psi\circ\Phi)_t & (\Psi\circ\Phi)_y\\
(\Psi\circ\Phi)_{tt} & (\Psi\circ\Phi)_{ty}
\end{pmatrix}.
\ee

To ensure that $\T_\s$ is locally smooth, we should be able to solve \eqref{ts-eqns} for $t$ and $y_1$ in terms of $y_\perp$. Due to  $(\Psi\circ\Phi)_{y_\perp}=0$ (cf. \eqref{y-coords}), this is the only possible split of the $y$ coordinates. Hence, \eqref{jacob-1} and $(\Psi\circ\Phi)_t=0$ lead to the requirement that $\det(\Psi\circ\Phi)_{tt}\not=0$. Then, solving \eqref{ts-eqns} determines $t=T^*(y_\perp)$ and $y_1=Y_1^*(y_\perp)$ as smooth functions of $y_\perp$ in a neighborhood of $y_\perp=0$. In particular, $y_1=Y_1^*(y_\perp)$ is a local equation of the codimension 1 submanifold $\T_\s\subset \vs$. The point of tangency $x^*=\Phi(T^*(y_\perp),(Y_1^*(y_\perp),y_\perp))$ also depends smoothly on $y_\perp$.

To find $\T_{x_0}$, solve $x_0=\Phi(t,y)$ for $t$ and a subset of the $y$ variables. Split the $y$ coordinates further so that 
\be\label{y-split-extra}
y=\bma \yo\\ \yt \ema,\ \yo\in\br^{n-N},\ \yt\in\br^{N},\ \Phi^\os_{\yt}=0,\ \det\Phi_{(t,\yo)}\not=0.
\ee
The property $\det\Phi_{(t,\yo)}\not=0$ does not generally hold, so additional assumptions are required (see assumptions G1 and G2 below). By \eqref{y-coords}, $y_1$ cannot be a part of $\yt$. Thus, similarly to the $x$-coordinates, we make $y_1$ to be the first coordinate of the first group $\yo$, i.e. $y_1=\yo_1$.

Equations \eqref{y-coords} and \eqref{y-split-extra} form a complete set of requirements that the new $y$ coordinates are supposed to satisfy. The required orthogonal matrix $U$ can be found as follows. Let $V_1\Sigma V_2^T$ be the SVD of the Jacobian matrix $\Por_{\yor}^{(1)}(t=0,\yor_0)$. Here $V_1\in O(n-N)$ and $V_2\in O(n)$ are orthogonal matrices, and $\Sigma$ is a rectangular $(n-N)\times n$ matrix with $\Sigma_{ij}=0$, $i\not=j$, and $\Sigma_{ii}>0$,  $1\le i \le n-N$. The latter property follows from assumptions G1, G2, and $\Phi^\os_t=0$, which yield that $\text{rank}\Por^\os_{\yor}=n-N$. Then we can take $U=V_2$. Indeed,
\be\label{U-deriv}
\frac{\pa\Phi^{(1)}}{\pa\yt}=\frac{\pa\Por^{(1)}}{\pa\yor}\frac{\pa\yor}{\pa\yt}
=V_1\Sigma V_2^T V_2^{\tp}=V_1\Sigma\bma 0\\I_N\ema=0.
\ee
Here $V_2^{(2)}$ is the $n\times N$ matrix consisting of the last $N$ columns of $V_2$, and $I_N$ is the $N\times N$ identity matrix.
Likewise,
\be\label{U-deriv-v2}
\frac{\pa\Phi^{(1)}}{\pa\yo}=\frac{\pa\Por^{(1)}}{\pa\yor}\frac{\pa\yor}{\pa\yo}
=V_1\Sigma V_2^T V_2^{\os}=V_1\Sigma\bma I_{n-N}\\0\ema,\
\det \frac{\pa\Phi^{(1)}}{\pa\yo}\not=0.
\ee
Together with $\Phi^\os_t=0$ and $\det\Phi^\tp_t\not=0$ this implies that $\det\Phi_{(t,\yo)}\not=0$ (cf. \eqref{y-split-extra}). 

Solve $x_0=\Phi(t,y)$ for $(t,\yo)$ to get an equation for $\T_{x_0}$ in the form $y=Y_0(\yt)$. The property $\det\Phi_{(t,\yo)}\not=0$ implies that $\T_x\subset\vs$ is an embedded codimension $n-N$ manifold for any $x\in\us$ (recall that both $\us$ and $\vs$ are sufficiently small). 

Introduce the matrix
\be\label{Bolker-matr}
M:=\begin{pmatrix}\Phi_t & \Phi_y\\
(\xi_0\cdot\Phi)_{tt} & (\xi_0\cdot\Phi)_{ty}\end{pmatrix},\ \xi_0:=d_x\Psi\in T_{x_0}^*\us,
\ee
which is the Jacobian matrix for the equations
\be\label{Jacob-Bolker}
\Phi(t,y)=x_0,\ \xi_0\cdot\Phi_t(t,y)=0,
\ee
where $(t,y)\in\br^N\times\vs$ are the unknowns. We assume that $M$ is non-degenerate (this is the Bolker condition). This condition guarantees that any singularity of $f$ microlocally near $(x_0,\xi_0)$ is visible in the GRT data $\R f(y)$, $y\in\vs$. In block form
\be\label{Bolker-block}\begin{split}
&M=\begin{pmatrix}M_{11} & M_{12}\\
M_{21} & M_{22}\end{pmatrix},\ M_{11}\in\br^{n\times n},  M_{12}\in\br^{n\times N},
M_{21}\in\br^{N\times n},  M_{22}\in\br^{N\times N},\\
&M_{11}=\Phi_{(t,\yo)},\ M_{12}=\Phi_{\yt}.
\end{split}
\ee 
In the selected $x$- and $y$-coordinates (see \eqref{xcoords}, \eqref{y-split-extra}), the matrix $M$ becomes
\be\label{M-alt}\begin{split}
&M=\begin{pmatrix} 0 & \Phi^{\os}_{\yo} & 0\\
\Phi^\tp_t & \Phi^{\tp}_{\yo} & \Phi^\tp_{\yt} \\
\xi_0\cdot\Phi_{tt} & \xi_0\cdot\Phi_{t\yo} & \xi_0\cdot\Phi_{t\yt}
\end{pmatrix}.
\end{split}
\ee 
In what follows, we use
\be\label{Theta-def}
\bp:=d_y(\Psi\circ\Phi)=\Phi^*\xi_0,
\ee
i.e. $\bp\in T_{y_0}^*\vs$ is the pull-back of $\xi_0\in T_{x_0}^*\us$ by $\Phi(0,\cdot)$. By \eqref{y-coords}, $\bp=dy_1$. 

To summarize, our main assumptions describing the geometry of the GRT.

{\bf Geometric assumptions.}
\begin{enumerate}
\item[G1.] $\text{rank}\, \Phi_t=N$;
\item[G2.] $\text{rank}\, \Phi_{(t,y)}=n$;
\item[G3.] $\det M\not=0$;
\end{enumerate}
From the above assumptions it follows that 
\be\label{conseqs}
\det \Phi^{\os}_{\yo}\not=0,\ \det M_{11}=\det \Phi_{(t,\yo)}\not=0,\ \det \Phi^{\tp}_t\not=0.
\ee

\begin{definition}\label{def:gen} A pair $(x_0,\yor_0)$, $x_0\in\us$, $\yor_0\in \tilde \T_{x_0}$,
is generic for the sampling matrix $D$ if 
\begin{enumerate}
\item There is no vector $m\in\mathbb Z^n$ such that the 1-form $D^{-T}md\yor$ vanishes identically on the tangent space to $\widetilde\T_{x_0}$ at $\yor_0$, and
\item $\dsf$ is either positive definite or negative definite.
\end{enumerate}
\end{definition}

Equivalently, we can define $(x_0,\xi_0)\in T^*\us$ to be generic if for every $\yor_0\in \tilde \T_{x_0}$ such that $\xi_0$ is conormal to $\s_{\yor_0}$ the pair $(x_0,\yor_0)$ is generic in the sense of  Definition~\ref{def:gen}.

In the rest of the paper, we assume that the pair $(x_0,\yor_0)$ is generic in the sense of  Definition~\ref{def:gen}, and $\dsf$ is {\it negative} definite. The latter assumption is not restrictive, because the positive and negative definite cases can be converted into each other by a change of the $x$ coordinates.

Sometimes we assume that $g$ is a conormal distribution associated with a smooth hypersurface $\Gamma\subset\vs$ regardless of whether there is an $f$ such that $g=\R f$. If there are a pair $(x_0,y_0)\in\us\times\vs$ and $\xi_0\in T_{x_0}^*\us$ such that (a) $\T_{x_0}$ is tangent to $\Gamma$ at $y_0$, and (b) $\xi_0\Phi'=(0,\Theta)$, where $(y_0,\Theta)\in N^*\Gamma$ (here $\Phi':=\Phi_{(t,y)}$ and $N^*\Gamma$ is the conormal bundle of $\Gamma$), then conditions G1--G3 imply that there is a smooth hypersurface $\s\subset\us$ such that $\Gamma=\T_\s$. This surface is found by solving $\xi\Phi'(t,y)=(0,\eta)$ for $\xi$ and $t$ in terms of $(y,\eta)\in N^*\Gamma$ and then setting $x=\Phi(t,y)$, where $t=t(y,\eta)$. The matrix $M$ in \eqref{Bolker-matr} is the Jacobian of the equation, where $(\xi,t)\in (\br^n\setminus\{0\})\times\br^N$ are the unknowns. Once $\s$ is found, a function $\Psi\in C^\infty(\us)$ such that $\Psi(x)=0$ if and only if $x\in\s$ and $\Psi'(x)\not=0$ for any $x\in\us$ can be easily found as well. 

\section{Tangency of $\T_\s$ and $\T_{x_0}$}\label{sec:tang}
In this section we show that $\T_\s$ and $\T_{x_0}$ are tangent at $y_0$, and investigate their properties near the point of tangency.

\begin{lemma}\label{lem:tang_0} The submanifolds $\T_\s$ and $\T_{x_0}$ are tangent at $y_0=0$, and 
$\bp=d_y(\Psi\circ\Phi)$ is conormal to both of them at $y_0=0$.
\end{lemma}
\begin{proof}
Begin with $\T_\s$. Viewing $t$ and $y_1$ as functions of $y_\perp$, differentiating the equation $(\Psi\circ\Phi)(t,y)=0$ (cf. \eqref{ts-eqns}) gives:
\be\label{normal-dir-prf}
d\Psi\cdot(\Phi_t t_{y_\perp}+\Phi_{y_1} (y_1)_{y_\perp}+\Phi_{y_\perp})=
(\Psi\circ\Phi)_{y_1} (y_1)_{y_\perp}=0,
\ee
which implies $(y_1)_{y_\perp}=0$. Here we have used that (cf. \eqref{y-coords}, \eqref{xcoords})
\be\label{useful-assns}
(\Psi\circ\Phi)_t=0,\  (\Psi\circ\Phi)_{y_1}\not=0,\  (\Psi\circ\Phi)_{y_\perp}=0.
\ee
Therefore, in the selected $y$ coordinates, the equation of the tangent space $T_{y_0}\T_\s$ (viewed as a subspace of $\br^n$) is $y_1=0$. By the last equation in \eqref{useful-assns},
\be\label{dpsiphi}
0=(\Psi\circ\Phi)_{y_j}=d_y(\Psi\circ\Phi)e_j=\bp e_j,\ j=2,\dots,n,
\ee
which implies that $\Theta$ is conormal to $\T_\s$ at $y_0$. Here $e_j$ is the standard $j$-th basis vector in $\br^n$, and $y_\perp=(y_2,\dots,y_n)^T$.

Consider next $\T_{x_0}$. Differentiating $x_0^{\os}=\Phi^\os(t,y)$, where $t=t(\yt)$ and $\yo=Y_0(\yt)$, and using \eqref{xcoords}, \eqref{y-split-extra}, and \eqref{conseqs} gives $\pa \yo/\pa\yt=0$. Hence the equation of the tangent space $T_{y_0}\T_{x_0}$ is $\yo=0$, i.e. $T_{y_0}\T_{x_0}$ is a subspace of $T_{y_0}\T_\s$.
\end{proof}

Next we look more closely at the contact between $\T_{x_0}$ and $\T_\s$. 

\begin{lemma}\label{lem:tang} Let $y=Y_0(\yt)$ be the equation of $\T_{x_0}$. Let $y=Z(\yt)$ be the projection of $Y_0(\yt)$ onto $\T_\s$ along the first coordinate, i.e. $Z(\yt)\in\T_\s$ and $Y_0(\yt)-Z(\yt)=(h(\yt),0,\dots,0)^T$ for a scalar function $h(\yt)$. Then, with $M$ as in \eqref{Bolker-block}, and $\bp$ as in \eqref{Theta-def}, we have
\be\label{YZ-diff}\begin{split}
&\bp\cdot(Y_0(\yt)-Z(\yt))=-\frac12Q\yt\cdot\yt+O(|\yt|^3),\\
&Q=C^T(\Psi\circ\Phi)_{tt}^{-1}C,\ C=M_{22}-M_{21}M_{11}^{-1}M_{12},\ \det C\not=0.
\end{split}
\ee

\end{lemma}

\begin{proof} We solve separately two sets of equations:
\be\label{two-syst}\begin{split}
&\T_{x_0}:\,\Phi(t,y)=0;\\ 
&\T_{\s}:\,(\Psi\circ\Phi)(t,y)=0,\ (\Psi\circ\Phi)_t(t,y)=0.
\end{split}
\ee
Since $(t=0,y=0)$ solves \eqref{two-syst}, we have to first order in $\yt$
\be\label{frst-order}\begin{split}
&\T_{x_0}:\,\Phi_t t+\Phi_yy=0;\\ 
&\T_{\s}:\,(\Psi\circ\Phi)_t \check t+(\Psi\circ\Phi)_y y=0,\ (\Psi\circ\Phi)_{tt}\check t+ (\Psi\circ\Phi)_{ty}y=0.
\end{split}
\ee
The solution to the second system (i.e., related to $\T_\s$) is denoted with a check. Because $\T_{x_0}$ and $\T_\s$ are tangent at $y_0$, $y$ is the same in both solutions to first order in $\yt$. Recall that we search not for the general solution $y\in\T_\s$, but for points $y=Z(\yt)$ obtained by  projecting $Y(\yt)\in\T_{x_0}$ onto $\T_\s$ along $y_1$.
By Lemma~\ref{lem:tang_0}, $y_1=0$. Let $\Delta t$ and $\Delta y=\bma \Delta\yo \\ 0\ema$ denote second order perturbations. Since $\yt$ is an independent variable, its perturbation is not considered. Then, to second order in $\yt$,
\be\label{scnd-order}\begin{split}
&\T_{x_0}:\,\Phi_t \Delta t+\Phi_y\Delta y+\frac12\left(\Phi_{tt}t\cdot t+2\Phi_{ty}t\cdot y+\Phi_{yy}y\cdot y\right)=0;\\ 
&\T_{\s}:\,(\Psi\circ\Phi)_t \Delta \check t+(\Psi\circ\Phi)_y \Delta \check y\\
&\qquad
+\frac12\left((\Psi\circ\Phi)_{tt}\check t\cdot \check t+2(\Psi\circ\Phi)_{ty}\check t\cdot  y+(\Psi\circ\Phi)_{yy} y\cdot  y\right)=0.
\end{split}
\ee
Only the first equation on the second line of \eqref{frst-order} was used. Using \eqref{useful-assns}, the last equation in \eqref{two-syst}, and the first equation in \eqref{frst-order} in \eqref{scnd-order} gives
\be\label{scnd-order-2}\begin{split}
\T_{x_0}:\,(d\Psi\cdot\Phi_{y_1})\Delta y_1=&-\frac12\left((d\Psi\cdot\Phi_{tt})t\cdot t+2(d\Psi\cdot\Phi_{ty})t\cdot y+(d\Psi\cdot\Phi_{yy})y\cdot y\right)\\ 
=&-\frac12\left((\Psi\circ\Phi)_{tt} t\cdot  t+2(\Psi\circ\Phi)_{ty} t\cdot  y+(\Psi\circ\Phi)_{yy} y\cdot  y\right);\\
\T_{\s}:\,(d\Psi\cdot\Phi_{y_1}) \Delta \check y_1=&
-\frac12\left((\Psi\circ\Phi)_{tt}\check t\cdot \check t+2(\Psi\circ\Phi)_{ty}\check t\cdot  y+(\Psi\circ\Phi)_{yy} y\cdot  y\right).
\end{split}
\ee
Subtracting the two equations and using the last equation in \eqref{frst-order} gives
\be\label{diff}
\Delta y_1-\Delta \check y_1=\bp\cdot(\Delta y-\Delta \check y)
=-\frac12(\Psi\circ\Phi)_{tt}(t-\check t)\cdot (t-\check t).
\ee
Using again the last equation in \eqref{frst-order} and then the first:
 \be\label{Phi-dtt}
(\Psi\circ\Phi)_{tt}(t-\check t)=(\Psi\circ\Phi)_{tt}t+(\Psi\circ\Phi)_{ty}y
=\xi_0\cdot (\Phi_{tt}t+\Phi_{ty}y),\ \xi_0=d_x\Psi.
\ee
The map $\yt\to t-\check t$ can be computed explicitly. The equation for $\T_{x_0}$ in \eqref{frst-order} and \eqref{Bolker-block} imply
\be\label{frst-eq}
 \begin{pmatrix}t \\ \yo \end{pmatrix} =-M_{11}^{-1}M_{12} \yt.
\ee 
Therefore, by \eqref{Bolker-matr}, \eqref{Bolker-block}, and \eqref{Phi-dtt}
\be\label{Phi-dtt-2nd}
(\Psi\circ\Phi)_{tt}(t-\check t)
=\left(M_{22}-M_{21}M_{11}^{-1}M_{12}\right)\yt=C\yt.
\ee
As is easily checked, 
\be\label{M-deriv}
M\bma M_{11}^{-1} & 0\\ 0 & I\ema \bma I & -M_{12} \\ 0 & I\ema=\bma I & 0\\ * & C\ema,
\ee
where $I$ denotes identity matrices of various sizes. Therefore, $\det M/\det M_{11}=\det C$. Assumption G3 and \eqref{conseqs} imply $\det C\not=0$. All the assertions in \eqref{YZ-diff} now follow.
\end{proof}


Set $x_\e:=x_0+\e\check x$, and let $y=Y(\yt,x_\e)$ be a local equation for $\T_{x_\e}$ in a neighborhood of $y_0=0$. This equation is obtained by solving $\e\check x=\Phi(t,y)$. Suppose $|\check x|=O(1)$ and $|\yt|=O(\e^{1/2})$. The term $\e\check x$ is of a lower order, so the equation for $\T_{x_0}$ in \eqref{frst-order} is accurate on $\T_{x_\e}$ to the order $\e^{1/2}$. 
The updated version of the top equation in \eqref{scnd-order} becomes
\be\label{xe-pert}
M_{11}\bma \Delta t \\ \Delta \yo\ema
+\frac12\left(\Phi_{tt}t\cdot t+2\Phi_{ty}t\cdot y+\Phi_{yy}y\cdot y\right)=\e\check x,
\ee
which is correct to order $\e$. The terms $\frac12(\cdot)$ on the first line of \eqref{scnd-order} and in \eqref{xe-pert} are the same. Therefore, to order $\e$, introduction of the term $\e\check x$ requires only a linear correction compared with $Y_0(\yt)=Y(\yt,x_0)$, and we have
\be\label{frst-order-eps}
Y(\yt,x_\e)=Y_0(\yt)+\e \bma \mathcal P_{\yo}M_{11}^{-1}\check x\\ 0\ema +O(\e^{3/2}),\ 
|\yt|=O(\e^{1/2}),
\ee
where $\mathcal P_{\yo}$ is the projection $\bma t\\ \yo\ema\to\yo$. 
From \eqref{Bolker-block}, $M_{11}^{-1}=\pa (t,\yo)/\pa x$, where $t$ and $\yo$ are functions of $x$ obtained by solving $x=\Phi(t,(\yo,\yt=0))$, the derivative is evaluated at $x=0$, and 
\be\label{PM11}
\mathcal P_{\yo}M_{11}^{-1}=\left.\frac{\pa\Yo(\yt,x)}{\pa x}\right|_{\yt=0,x=0}=\Yo_x. 
\ee

%

\section{Behavior of the GRT near its singular support}\label{GRT-behavior}

To simplify notations, in the rest of the paper we set $b(x,y):=b(x,\yor(y))$, $w(x,y):=w(x,\yor(y))$,  and $g(y):=g(\yor(y))$. The original versions of these functions are used only in Section~\ref{sec:a-prelims}.

Given an open set $U\subset\br^n$, $r\in\br$, and $k\in\N$, by $S^r(U\times\br^k)$ we denote the set of all $h\in C^\infty(U\times(\br^k\setminus \{0\}))$ that satisfy 
\be\label{symbol-ineqs}
|\pa_x^{m}\pa_\xi^{\al} h(x,\xi)|\le c_{m,\al}|\xi|^{r-|\al|},\ x\in U,|\xi|\ge 1,
\ee
for any $m\in\N_0^n$, $\al\in \N_0^k$, $\N_0:=\N\cup\{0\}$, and some $c_{m,\al}$.

Suppose $f\in\CE'(\us)$ is given by
\be\label{f-orig}
f(x)=\frac1{2\pi}\int \tilde f(x,\la)e^{-i\la\Psi(x)}d\la,
\ee
where $\Psi$ is the same as in Sections~\ref{sec:prelims}, \ref{sec:tang}, and $\tilde f$ satisfies
\be\label{f-lim}\begin{split}
&\tilde f(x,\la)= \tilde f^+(x)\la_+^{-(s_0+1)+\frac N2}
+\tilde f^-(x)\la_-^{-(s_0+1)+\frac N2}+\tilde R(x,\la),\ x\in \us,\la\not=0;\\ 
&\pa_x^{m} \tilde f \in L_{loc}^1(\us\times\br),m\in\N_0^n;\, \tilde R\in S^{-(s_1+1)+\frac N2}(\us\times\br),\, \tilde f^\pm\in\coi(\us),\, 0< s_0<s_1.
\end{split}
\ee 
We use the superscripts $'\pm'$ to distinguish between two different functions as opposed to the positive and negative parts of a number. The latter are denoted by the subscripts $'\pm'$: $\la_\pm:=\max(\pm \la,0)$.

The GRT of $f$ is given by
\be\label{GRT-f}\begin{split}
\R f(y) &=\int_{\br^N} f(x) b(x,y)(\det G^{\s}(t,y))^{1/2}dt\\
&=\frac1{2\pi}\int_\br\int_{\br^N} \tilde f(x,\la)e^{-i\la\Psi(x)}b(x,y)(\det G^{\s}(t,y))^{1/2}dt d\la,\ x=\Phi(t,y).
\end{split}
\ee
Here $G^{\s}$ is the Gram matrix
\be\label{gram-1st}
G_{jk}^{\s}(t,y)=\frac{\pa \Phi(t,y)}{\pa t_j}\cdot\frac{\pa \Phi(t,y)}{\pa t_k},\ 
1\le j,k\le N,
\ee
so $(\det G^{\s}(t,y))^{1/2}dt$ is the volume form on $\s_y$.

Consider the second equation for $\T_\s$ in \eqref{two-syst} and solve it for $t$. Since $\det (\Psi\circ\Phi)_{tt} \not=0$, the solution $t^*=t^*(y)$ is a smooth function. The function $t^*(y)$ here is different from $T^*(y_\perp)$ in the paragraph following \eqref{jacob-1}, because now we solve only the second of the two equations that define $\T_\s$. The asymptotics as $\la\to\infty$ of the integral with respect to $t$ in \eqref{GRT-f} is computed with the help of the stationary phase method (see (2.14)--(2.16) in Chapter VIII of \cite{trev2}) 
\be\label{GRT-f-inner}\begin{split}
\int &\tilde f(x,\la)e^{-i\la\Psi(x)}b(x,y)(\det G^{\s}(t,y))^{1/2}dt\\
=&\left(\tilde f(x^*,\la)b(x^*,y)\left|\frac{\det G^{\s}(t^*,y)}{\det (\Psi\circ\Phi)_{tt}(t^*,y)}\right|^{1/2}
\left(\frac{2\pi}{|\la|}\right)^{N/2}+\tilde R(y,\la)\right)\\
&\times e^{-i\frac\pi4\text{sgn}(\la (\Psi\circ\Phi)_{tt}(t^*,y))}
e^{-i\la\Psi(x^*)},\ x^*=\Phi(t^*(y),y),\ \tilde R\in S^{-(s_0+2)}(\vs\times\br).
\end{split}
\ee

Introduce the function
\be\label{Pdef}
P(y):=(\Psi\circ\Phi)(t^*(y),y).
\ee
Then $\R f$ can be written as
\be\label{GRT-f-alt}
\R f(y) =\frac1{2\pi}\int \fs(y,\la)e^{-i\la P(y)}d\la,
\ee
and
\be\label{f-lim-coefs}\begin{split}
\fs(y,\la)=& \fs^+(y)\la_+^{-(s_0+1)}
+\fs^-(y)\la_-^{-(s_0+1)}+\tilde R(y,\la);\
\pa_y^m\tilde v\in L^1_{loc}(\vs\times\br), m\in\N_0^n;\\ 
\fs^\pm\in&\coi(\vs),\ \tilde R\in S^{-\min(s_1+1,s_0+2)}(\vs\times\br),\ 0< s_0<s_1;\\
\fs^\pm(y)=&(2\pi)^{N/2}\tilde f^\pm(x^*)b(x^*,y)\left|\frac{\det G^{\s}(t^*,y)}{\det (\Psi\circ\Phi)_{tt}(t^*,y)}\right|^{1/2} e^{\mp i\frac\pi4\text{sgn}( (\Psi\circ\Phi)_{tt}(t^*,y))}.
\end{split}
\ee 

By construction, $P(y)=0$ is another equation for $\T_\s$. Since $(\Psi\circ\Phi)_t=0$, the first equation for $\T_\s$ in \eqref{frst-order} does not determine $t^*$. Therefore, to first order, $t^*$ is determined by solving the last equation in \eqref{frst-order}:
\be\label{tstar-soln}
t^*(y)=-(\Psi\circ\Phi)_{tt}^{-1}(\Psi\circ\Phi)_{ty}y+O(|y|^2),
\ee
and
\be\label{P-exp}
P(y)=d_y(\Psi\circ\Phi)y +O(|y|^2)=\bp\cdot y+O(|y|^2).
\ee

\section{Setting of the reconstruction problem, main assumptions}\label{sec:assns}

In view of the notation convention stated at the beginning of Section~\ref{GRT-behavior}, the reconstruction is given by 
\be\label{recon_cont}
(\R^*\CB g)(x)=\int_{\T_x}(\CB g)(y) w(x,y)dy,\ g=\R f,
\ee
where $dy$ is the volume form on $\T_x$, $\CB$ is a $\Psi$DO
\be\label{psDO}
(\CB g)(y):=\frac1{(2\pi)^n}\int \tilde B(y,\eta) \tilde g(\eta)e^{-iy\cdot\eta}d\eta,\ 
\tilde g=\CF g,
\ee
and $\CF$ is the Fourier transform in $\br^n$. Using \eqref{data-pts-or} and that $\yor=Uy+\yor_0$, the discrete data $g(\vy^j)$ are known at the points 
\be\label{data-pts}
\vy^j=U^T(\e Dj-\yor_0),\ j\in\mathbb Z^n.
\ee
Reconstruction from discrete data is given by 
\be\label{recon}
\check f_\e(x)=(\R^*\CB g_\e)(x)=\int_{\T_x}(\CB g_\e)(y) w(x,y)dy,
\ee
where $g_\e(y)$ is the interpolated data: 
\be\label{interp-data}
g_\e(y):=\sum_j \ik((y-\vy^j)/\e)g(\vy^j),\ \vy^j=U^T(\e Dj-\yor_0),\ j\in\mathbb Z^n,
\ee
and $\ik$ is an interpolation kernel. Now we state all the assumptions about $\CB$, $\ik$, and $g$.

{\bf Assumptions about $\CB$:}
\begin{itemize} 
\item[$\CB1$.] $\tilde B(y,\eta)\equiv0$ outside a small conic neighborhood of  $(y_0,\bp)$;
\item[$\CB2$.] The amplitude of $\CB$ satisfies
\be\label{B-lim}\begin{split}
&\pa_y^m\tilde B\in L^\infty_{loc}(\vs\times\br^n), m\in\N_0^n;\ \tilde B-\tilde B_0\in S^{\bt_1}(\vs\times\br^n);\\
&\tilde B_0\in C^{\infty}(\vs\times (\br^n\setminus \{0\})),\ \tilde B_0(y,\la\eta)=\la^{\bt_0}\tilde B_0(y,\eta), \la>0;\ \bt_0>\bt_1.
\end{split}
\ee 
\end{itemize}

Recall that $\lfloor r\rfloor$, $r\in\br$, denotes the largest integer not exceeding $r$. Similarly,  $\lceil r\rceil$ denotes the smallest integer greater than or equal to $r$. We also introduce:
\be\label{flclpm}
\lfloor r^-\rfloor=\lim_{\e\to+0}\lfloor r-\e\rfloor=\begin{cases}\lfloor r\rfloor,&r\not\in\mathbb Z,\\ r-1,&r\in\mathbb Z,\end{cases}\ \lceil r^+\rceil=\lim_{\e\to+0}\lceil r+\e\rceil=\begin{cases}\lceil r\rceil,&r\not\in\mathbb Z,\\ r+1,&r\in\mathbb Z.\end{cases}
\ee

{\bf Assumptions about the interpolation kernel $\ik$:}
\begin{itemize}
\item[IK1.]\label{ikcont} $\ik\in C_0^{\lceil \bt_0^+\rceil}(\br^n)$, i.e. $\ik$ is compactly supported and all of its derivatives up to order $\lceil \bt_0^+\rceil$ are bounded;
\item[IK2.] $\ik$ is exact up to order $\lceil \bt_0\rceil$, i.e.
\be\label{exactness}
\sum_{j\in\mathbb Z^n} j^m\ik(u-j)\equiv u^m,\ |m|\le \lceil \bt_0\rceil,\ m\in\N_0^n,\ u\in\br^n.
\ee
\end{itemize}

Assumption IK2 with $m=0$ implies that $\ik$ is normalized:
\be\label{norm-deriv}
1=\int_{[0,1]^n}\sum_j\ik(u-j)du=\int_{\br^n}\ik(u)du.
\ee

Inspired by \eqref{GRT-f-alt}, assume that $g$ is given by
\be\label{g-def}
g(y)=\frac1{2\pi}\int_{\br}\fs(y,\la)e^{-i\la P(y)}d\la,\ |P'(y)|\not=0,\,y\in\vs,
\ee
and $g$ is sufficiently regular. The smooth hypersurface determined by $P$ is
\be\label{surf-def}
\Gamma:=\{y\in\vs:\, P(y)=0\}.
\ee
In terms of our coordinates, $P$ can be selected in the form (cf. \eqref{y-coords}, \eqref{Theta-def},  and \eqref{P-exp})
\be\label{P-expl}
P(y)=y_1-\psi(y_\perp)
\ee
for a smooth $\psi$ that satisfies $\psi(0)=0$ and $\psi'(0)=0$. Thus, $dP(0)=\bp$.

{\bf Assumptions about $g$:}
\begin{itemize}
\item[g1.] $\pa_y^m\fs\in L^1(\vs\times\br)$, $m\in\N_0^n$, and there exists a compact $K\subset\vs$ such that $\fs(y,\la)\equiv0$ if $y\in\vs\setminus K$;
\item[g2.] $\fs$ satisfies
\be\label{g-lim}\begin{split}
&\fs(y,\la)= \fs^+(y)\la_+^{-(s_0+1)}
+\fs^-(y)\la_-^{-(s_0+1)}+\tilde R(y,\la),\\ 
&\fs^\pm\in\coi(\vs),\ \tilde R\in S^{-(s_1+1)}(\vs\times\br),\ 0< s_0<s_1,\ s_1\not\in\N;
\end{split}
\ee 
\item[g3.] $P(y)$ is given by \eqref{P-expl}, where $\psi$ is smooth, $\psi(0)=0$, $\psi'(0)=0$, and $P'(y)\not=0$ on $\vs$;
\item[g4.] If $s_0\in\mathbb N$, one has 
\be\label{s-except}
\fs^+(\bar y)=(-1)^{s_0+1}\fs^-(\bar y),\ \bar y=(\psi(y_\perp),y_\perp)\in \Gamma.
\ee
\end{itemize}
The assumption $s_1\not\in\N$ is not restrictive. It is made to simplify some of the proofs.

Define
\be\label{e-def}
e(a):=\exp\left(i\frac\pi2 a\right).
\ee
Requirements that combine the properties of $\CB$ and $g$ are the following
\begin{itemize}
\item[C1.] One has
\be\label{assump}
\kappa:=\bt_0-s_0-(N/2)\ge 0;
\ee
\item[C2.] One has
\be\label{c-ratios-req}\begin{split}
&\tilde B_0(\bar y,P'(\bar y))\fs^+(\bar y)=-e(2(\bt_0-s_0))\tilde B_0(\bar y,-P'(\bar y))\fs^-(\bar y)\quad \text{if}\quad\kappa=0,\\
&\bar y:=(\psi(y_\perp),y_\perp)\in \Gamma.
\end{split}
\ee
\end{itemize}

%

The meaning of conditions \eqref{s-except} and \eqref{c-ratios-req} is that they prevent the appearance of logarithmic terms in $g$ and $\check f=\R^*\CB g$ in a neighborhood of $\Gamma$ and $\s$, respectively, see \eqref{g-Psi_N} and \eqref{CTB-st3}.


 

Substitute \eqref{interp-data} into \eqref{recon} and use that $g$ is compactly supported:
\be\label{recon-d}
\check f_\e(x_\e)= \int_{y\in\T_{x_\e}}\sum_{|j|\le O(1/\e)} \left(\CB \ik\left(\frac{\cdot-\vy^j}\e\right)\right)(y) g(\vy^j)w(x_\e,y)dy.
\ee
Notation $|j|\le O(1/\e)$ means that $j$ satisfies $\vy^j\in\text{supp}(g)$, and the set of all such $j$ is contained in a ball of radius $O(1/\e)$. Notation $\CB \ik\left(\frac{\cdot-\vy^j}\e\right)$ is understood as follows:
\be\label{B-expl}
\left(\CB \ik\left(\frac{\cdot-\vy^j}\e\right)\right)(y)=(\CB\ik_1)(y),\ \ik_1(z):=\ik\left(\frac{z-\vy^j}\e\right).
\ee

Pick some large $A\gg 1$ and introduce two sets
\be\label{two-sets}\begin{split}
\Omega_1:=&\left\{\yt\in\br^{N}:\, |\yt|\le A\e^{1/2}\right\},\\
\Omega_2:=&\left\{\yt\in\br^{N}:\, |\yt|\ge A\e^{1/2},\ \bma \yo\\ \yt \ema\in\vs\right\}.
\end{split}
\ee
Let $f_\e^{(l)}(x)$ denote the reconstruction obtained using \eqref{recon-d}, where the $y$ integration is restricted to the part of $\T_x$ corresponding to $\Omega_l$, $l=1,2$, respectively.

\section{On some properties of $g$ and $g_\e$}\label{sec:gprops}

By Proposition 25.1.3 in \cite{hor4}, $g$ 
is a conormal distribution with respect to $\Gamma$.
The wave front set of $g$ is contained in the conormal bundle of $\Gamma$: $WF(g)\subset N^*\Gamma=\{(y,\eta)\in \vs\times(\br^n\setminus \{0\}): P(y)=0,\eta=\la P'(y)\}$.  See also Section 18.2 and Definition 18.2.6 in \cite{hor3} for a formal definition and in-depth discussion of conormal distributions. A discussion of closely related Lagrangian distributions is in Section 25.1 of \cite{hor4}.


In this paper we use two types of spaces of continuous functions. First, $C_b^k(\br^n)$, $k\in \mathbb N_0$, is the Banach space of functions with bounded derivatives up to order $k$. The norm in $C_b^k(\br^n)$ is given by
\be\label{ck-def}
\Vert h\Vert_{C_b^k}:=\max_{|m|\le k}\vert h^{(m)}\vert_{L^\infty}.
\ee
The subscript `0' in $C_0^k$ means that we consider the subspace of compactly supported functions, $C_0^k(\br^n)\subset C_b^k(\br^n)$. 

To describe the second space, pick any $\mu_0\in\coi(\br)$ such that $\mu_0(\eta)=1$ for $|\eta|\le 1$, $\mu_0(\eta)=0$ for $|\eta|\ge 2$, and define $\mu_j(\eta):=\mu_0(2^{-j}\eta)-\mu_0(2^{-j+1}\eta)$, $j\in\mathbb N$. Then the Holder-Zygmund space $C_*^r(\br^n)$, $r>0$, is defined as follows
\be\label{hz-sp}
C_*^r(\br^n):=\{h\in C_b^0(\br^n):\,\Vert h\Vert_{C_*^r}<\infty\},\
\Vert h\Vert_{C_*^r}:=\sup_{j\in\mathbb N_0}2^{jr}\Vert \mu_j(d/dp)h\Vert_{L^\infty}.
\ee
If $r\not\in\mathbb Z$, then $C_*^r(\br^n)$ consists of continuous functions, which have Holder continuous $\lfloor r\rfloor$-th order derivatives:
\be\label{holder}
\max_{|m|=\lfloor r\rfloor}\sup_{x\in\br^n,|h|>0}\frac{|f^{(m)}(x+h)-f^{(m)}(x)|}{|h|^{\{r\}}}<\infty.
\ee
Here $\{r\}:=r-\lfloor r\rfloor$ is the fractional part of $r$. The Holder-Zygmund spaces are a particular case of the Besov spaces: $C_*^r(\br^n)=B^r_{p,q}(\br^n)$, where $p,q=\infty$ (see item 2 in Remark 6.4 of \cite{Abels12}). As is easily seen, $C_b^k\subset C_*^k$ if $k\in\mathbb N$.


The following two lemmas are proven in Appendix~\ref{sec:prfPc}.

\begin{lemma}\label{lem:conorm} Suppose $g$ satisfies the assumptions in Section~\ref{sec:assns}. There exist $c_m>0$ such that
\be\label{Psic-ass-g}
|\pa_y^m g(y)|\leq c_m \begin{cases} |P(y)|^{s_0-|m|},&|m|>s_0,\\
1,&|m|\le s_0,\end{cases}
\ m\in\N_0^n,\ y\in\vs\setminus\Gamma.
\ee
Additionally,
\be\label{g-continuity}
g\in C_*^{s_0}(\vs) \text{ and } g\in C_0^{s_0}(\vs) \text{ if }  s_0\in\mathbb N.
\ee
If the leading term in $\fs$ is missing, i.e. $\fs^\pm\equiv0$, then $g\in C_*^{s_1}(\vs)$, and \eqref{Psic-ass-g} holds with $s_0$ replaced by $s_1$. 
\end{lemma}

\begin{lemma}\label{lem:P_c-ass} Suppose $\CB$ and $g$ satisfy the assumptions in Section~\ref{sec:assns}. There exists $c_\bt>0$ such that
\be\label{Psic-ass-Bg}
|(\CB g)(y)|\leq c_\bt |P(y)|^{s_0-\bt_0},\ y\in\vs\setminus\Gamma.
\ee
If $\kappa=0$, we additionally have
\be\label{Psic-ass-extra}
|(\CB g)(y)|\leq c_\bt P(y)^{s_0-\bt_0+c},\ y\in\vs,\ P(y)>0,
\ee
for some $c>0$.
\end{lemma}

%

Define
\be\label{F-def}\begin{split}
g_\e^{(l)}(y):=&\sum_j\pa_{y_1}^l\ik\left(\frac{y-\vy^j}\e\right)g(\vy^j),\
\Delta g_\e^{(l)}(y):=g_\e^{(l)}(y)-\pa_{y_1}^l g(y),\ 0\le l\le \lceil \bt_0^+\rceil.
\end{split}
\ee
The following two lemmas are proven in Appendix~\ref{sec:prfPc}.

\begin{lemma}\label{lem:gel} Suppose $g$ and $\ik$ satisfy the assumptions in Section~\ref{sec:assns}. There exists $\varkappa_1>0$ such that 
\be\label{F-ineq-lot}
|g_\e^{(l)}(y)| \leq c\begin{cases}
|P(y)|^{s_0-l},& |P(y)|\geq \varkappa_1\e,\ s_0< l\leq \lceil \bt_0^+ \rceil,\\
\e^{s_0-l},& |P(y)|\leq \varkappa_1\e,\ s_0< l\leq \lceil \bt_0^+ \rceil,\\
1,& 0\le l\le s_0,
\end{cases}
\ y\in\vs,
\ee
for some $c>0$. 

If the top order term in $\fs$ is missing, i.e. $\fs^\pm\equiv0$, then \eqref{F-ineq-lot} holds with $s_0$ replaced by $s_1$ as long as $ l\leq \lceil \bt_0^+ \rceil$. 
\end{lemma}


\begin{lemma}\label{lem:delta} Suppose $g$ and $\ik$ satisfy the assumptions in Section~\ref{sec:assns}. Let $\varkappa_1$ be the same as in Lemma~\ref{lem:gel}. One has
\begin{align}\label{F-ineq-1}
&|\Delta g_\e^{(l)}(y)| \leq c\e|P(y)|^{s_0-1-l}, \ y\in\vs,\,|P(y)|\ge \varkappa_1\e,\,\lfloor s_0^-\rfloor\le l\le \lceil \bt_0^+ \rceil,\\
\label{F-ineq-2}
&|\Delta g_\e^{(l)}(y)| \leq c\e^{s_0-l}, \ y\in\vs,\ 0\le l\le \lfloor s_0^-\rfloor,
\end{align}
for some $c>0$. 

If the top order term in $\fs$ is missing, i.e. $\fs^\pm\equiv0$, then \eqref{F-ineq-1}, \eqref{F-ineq-2} hold with $s_0$ replaced by $s_1$ as long as $l\le \lceil\bt_0^+\rceil$. 
\end{lemma}

%

\section{Computing the first part of the leading term}\label{sec:fplt}

Throughout this section we assume that $\CB$ in \eqref{psDO} satisfies $\tilde B(y,\eta)\equiv\tilde B_0(y,\eta)$, i.e. we assume that the symbol of $\CB$ contains only the top order term. Let $\CB_0$ denote the $\Psi$DO of the form \eqref{psDO}, where $\tilde B(y,\eta)\equiv\tilde B_0(y_0,\eta)$. Likewise, we assume that the symbol of $g$ coincides with its top order term (cf. \eqref{g-lim}, \eqref{g-FT}, and \eqref{g-Psi_nN})
\be\label{g-lead-sing}
g(y)=a^+(y)P_+^{\gao}(y)+a^-(y)P_-^{\gao}(y),\ a^\pm\in\coi(\vs).
\ee
Combining \eqref{g-lead-sing}, \eqref{recon}, and \eqref{interp-data}, and using that $g$ is compactly supported gives
\be\label{recon-lead-o}\begin{split}
f_\e^{(1)}(x_\e)= \int_{\substack{y\in\T_{x_\e}\\ |\yt|\le A\e^{1/2}}}\sum_{|j|\le O(1/\e)} \left(\CB \ik\left(\frac{\cdot-\vy^j}\e\right)\right)(y) g(\vy^j)w(x_\e,y)dy.
\end{split}
\ee
Notation $|j|\le O(1/\e)$ means that $j$ satisfies $\vy^j\in\text{supp}(g)$, and the set of all such $j$ is contained in a ball of radius $O(1/\e)$.

We begin by investigating the sum in \eqref{recon-lead-o}. The key result is the following lemma (see Appendix~\ref{sec:lem3prf} for the proof).

\begin{lemma}\label{key-sum-lemma}
Suppose $y,z\in\vs$ satisfy
\be\label{zchy}
|y-y_0|\le c\e^{1/2},\ |y-z|\le c\e,\ z\in \Gamma,
\ee
for some $c>0$. One has 
\be\label{I-def-est}\begin{split}
\e^{\bt_0-s_0}&\sum_{|j|\le O(1/\e)} \left(\CB \ik\left(\frac{\cdot-\vy^j}\e\right)\right)(y) \left[a^+(\vy^j)P_+^{\gao}(\vy^j)+a^-(\vy^j)P_-^{\gao}(\vy^j)\right]\\
=&\sum_{j} \CB_0 \ik\left(\frac{y-\vy^j}{\e}\right) \CA\left(\bp\cdot \frac{\vy^j-z}{\e}\right)+O(\e^{\min(s_0,1)/2}),\ \e\to0, 
\end{split}
\ee
where the big-$O$ term is uniform with respect to $z,y$ satisfying \eqref{zchy}, and
\be\label{CA-def}
\CA(p):=a^+(y_0)p_+^{s_0}+a^-(y_0)p_-^{s_0},\ p\in\br.
\ee
Moreover, the left-hand side of \eqref{I-def-est} remains bounded as $\e\to0$ uniformly with respect to $z,y$ satisfying \eqref{zchy}.
\end{lemma}

On the second line in \eqref{I-def-est}, $\CB_0$ acts with respect to the rescaled variable $\check y=y/\e$. Since $\CB_0$ is shift-invariant, it is not necessary to represent its action in the form \eqref{B-expl}.


The next step is to use \eqref{I-def-est} in \eqref{recon-lead-o}:
\be\label{recon-lead-o1}\begin{split}
&\e^{\bt_0-s_0}f_\e^{(1)}(x_\e)
= J_\e(\check x)+O(\e^{(N+\min(s_0,1))/2}),\\
&J_\e(\check x):=\int_{\Omega_1}\sum_j \CB_0 \ik\left(\frac{Y(\yt,x_\e)-\vy^j}\e\right)
\CA\left(\bp\cdot \frac{\vy^j-Z(\yt)}\e\right)d\yt,\\
&a^\pm:=a^\pm(y_0)w(x_0,y_0)(\det G^{\T})^{1/2},\ x_\e:=x_0+\e\check x,\\
&G_{ij}^{\T}(\yt)=\frac{\pa Y_0(\yt)}{\pa \yt_i}\cdot\frac{\pa Y_0(\yt)}{\pa \yt_j},\ 
1\le i,j\le N,\ Y_0(\yt):=Y(\yt,x_0).
\end{split}
\ee
Here $G^{\T}$ is the Gram matrix, in which all derivatives are evaluated at $\yt=0$, and $Z(\yt)$ is obtained by projecting $Y_0(\yt)$ onto $\Gamma$ along $y_1$, see Lemma~\ref{lem:tang}. To clarify the use of indices in \eqref{recon-lead-o1}, when $\yt$ is viewed as part of $y$, then $\yt_j=y_{n-N+j}$, $1\le j\le N$. Thus, $(\det G^{\T}(\yt))^{1/2}d\yt$ is the volume form on $\T_{x_0}$. By \eqref{frst-order-eps} and Lemma~\ref{lem:tang}, 
\be\begin{split}\label{YZ-diff-appr}
|Y(\yt,x_\e)-Z(\yt)|&=|Y_0(\yt)+O(\e)-Z(\yt)|\\
&=|Y_0(\yt)-Z(\yt)|+O(\e)=O(|\yt|^2)+O(\e)=O(\e), 
\end{split}
\ee
so the conditions in \eqref{zchy} hold, and \eqref{I-def-est} applies. Compared with \eqref{CA-def}, here we modified the coefficients $a^\pm$ of $\CA$ to include the additional factor $w(x_0,y_0)(\det G^{\T})^{1/2}$. This definition is understood in the rest of the paper.

Introduce the operator 
\be\label{oper1d}
\CBa g:=\CF_{1d}^{-1}(\tilde b(\la)\tilde g(\la)),\ \tilde b(\la):=\tilde B_0(y_0,\bp)\la_+^{\bt_0}+\tilde B_0(y_0,-\bp)\la_-^{\bt_0},
\ee
where $g$ is sufficiently smooth and decays sufficiently fast, and $\CF_{1d}$ denotes the 1D Fourier transform. In view of \eqref{interp-data} and \eqref{recon-lead-o1}, introduce two auxiliary functions:
\be\label{pPsi-def}\begin{split}
\Mpsi(\check y,p):=&\sum_{j\in\mathbb Z^n} \CB_0 \ik(\check y-D_1 j)\CA(\bp\cdot D_1 j-p),\
D_1:=U^TD,\\ 
\MPsi(p):=&\int \CBa\hat\ik(\bp,p-q) \CA(q)dq=\CF_{1d}^{-1}(\tilde\ik(\la\bp)\tilde b(\la)\tilde \CA(\la)),\ p\in\br,
\end{split}
\ee
where $\tilde\ik=\CF\ik$, $\tilde\CA=\CF_{1d}\CA$, $\CBa$ acts with respect to the affine variable, and the hat denotes the classical Radon transform that integrates over hyperplanes:
\be\label{crt-def}
\hat\ik(\bp,p):=\int \ik(x)\de(\bp\cdot x-p)dx.
\ee
Both $\tilde b(\la)$ and $\tilde\CA(\la)$ are not smooth at $\la=0$, so the product $\tilde b(\la)\tilde \CA(\la)$ needs to be computed carefully, see the discussion between \eqref{J-ft-v2} and \eqref{DTB-v2}. As is easily checked, 
\be\label{pPsi-props}\begin{split}
&\Mpsi(\check y+D_1 m,p+\bp\cdot D_1 m)=\Mpsi(\check y,p),\ m\in\mathbb Z^n;\\
&\int_{[0,1]^n} \Mpsi(\check y+D_1 u,p+\bp\cdot D_1 u)du=\int_{\br^n} \CB_0 \ik(\check y-u)\CA(\bp\cdot u-p)du\\
&\hspace{5.1cm}=\MPsi(\bp\cdot \check y-p).
\end{split}
\ee
Substitute \eqref{pPsi-def} into \eqref{recon-lead-o1}
\be\label{recon-lead-2}\begin{split}
\e^{\bt_0-s_0}f_\e^{(1)}(x_\e)
= & \int_{\Omega_1}\Mpsi
\left(\frac{Y(\yt,x_\e)+U^T\yor_0}{\e},\bp\cdot\frac{Z(\yt)+U^T\yor_0}{\e}\right)d\yt\\
&+O(\e^{(N+\min(s_0,1))/2}).
\end{split}
\ee
To simplify and evaluate the expression in \eqref{recon-lead-2} we use \eqref{frst-order-eps}, \eqref{PM11}, and the first equation in \eqref{pPsi-props}:
\be\label{recon-lead-22}\begin{split}
&\e^{\kappa} f_\e^{(1)}(x_\e)=O(\e^{\min(s_0,1)/2})\\
&+ \int_{|\ytt|\le A}\Mpsi
\biggl(\Yo_x\check x+D_1 u_\e+O(\e^{1/2}),
\bp\cdot\left(-\frac{Y_0(\yt)-Z(\yt)}{\e}+D_1 u_\e\right)\biggr)d\ytt,
\end{split}
\ee
where
\be\label{nots}
u_\e:=\left\{\frac{D_1^{-1}(U^T\yor_0+Y_0(\yt))}{\e}\right\},\quad 
\yt=\e^{1/2}\ytt,
\ee
and $\{u\}$ denotes the fractional part of a vector (computed componentwise). 

By Lemma~\ref{lem:tang},
\be\label{scndform}
\bp\cdot(Y_0(\yt)-Z(\yt))=-\frac{Q\yt\cdot\yt}2+O(|\yt|^3).
\ee
Therefore,
\be\label{recon-lead-3}\begin{split}
&\e^{\kappa} f_\e^{(1)}(x_\e)
=O(\e^{\min(s_0,1)/2}) \\
&+ \int_{|\ytt|\le A}\Mpsi
\biggl(\Yo_x\check x+D_1 u_\e+O(\e^{1/2}),\frac{Q\ytt\cdot\ytt}2+\bp\cdot D_1 u_\e+O(\e^{1/2})\biggr)d\ytt.
\end{split}
\ee
Introduce an auxiliary function
\be\label{psi1}
\Mpsi_1(\check y,q;u):=\Mpsi(\check y+D_1 u,q+\bp\cdot D_1 u),\ \check y,u\in\br^n,\ q\in\br.
\ee
From \eqref{pPsi-props}, $\Mpsi_1(\check y,q;u+m)=\Mpsi_1(\check y,q;u)$, $m\in\mathbb Z^n$. Thus, the integrand in \eqref{recon-lead-3} can be written in the form:
\be\label{recon-lead-4}
\Mpsi_1\biggl(\Yo_x\check x+O(\e^{1/2}),\frac{Q\ytt\cdot\ytt}2+O(\e^{1/2}); \frac{D_1^{-1}(U^T\yor_0+Y_0(\e^{1/2}\ytt))}{\e}\biggr).
\ee
The assumption $|\yt|=O(\e^{1/2})$ implies
\be\label{Y0exp}
Y_0(\yt)=Y_0'(0)\yt+\frac{Y_0''(0)\yt\cdot\yt}2+O(\e^{3/2}),
\ee
where we have used that $Y_0(0)=0$. To simplify \eqref{recon-lead-4} we use the following result, which is proven in Appendix~\ref{sec:delphi}.

\begin{lemma}\label{lem:psi-incr} Pick any $c$, $0<c<\infty$. One has
\be\label{an-psi-pr-1}\begin{split}
&\Mpsi(\check y+\Delta \check y,p)-\Mpsi(\check y,p)= O(|\Delta \check y|^{1-\{\bt_0\}}),\ \Delta \check y\to0,\ |\check y|,|p|\le c;\\
&\Mpsi(\check y,p+\Delta p)-\Mpsi(\check y,p)= O(|\Delta p|^{\min(s_0,1)}),\ \Delta p\to0,\ |\check y|,|p|\le c;
\end{split}
\ee
and the two big-$O$ terms are uniform in $\check y$ and $p$ confined to the indicated sets. 
\end{lemma}

By Lemma~\ref{lem:psi-incr}, \eqref{recon-lead-4} equals to
\be\label{psi1-simpl}
\Mpsi_1\biggl(\Yo_x\check x,\frac{Q\ytt\cdot\ytt}2; \frac{D^{-1}\yor_0}\e+D_1^{-1}\left(\frac{Y_0'(0)\ytt}{\e^{1/2}}+\frac{Y_0''(0)\ytt\cdot\ytt}2\right)\biggr)+O(\e^{a/2}),
\ee
where $a=\min(1-\{\bt_0\},s_0,1)>0$. Thus, we need to compute the limit of the following integral as $\e\to0$:
\be\label{key-int}\begin{split}
&J(\e):=\int_{|\ytt|\le A}\Mpsi_1(\check y,q;u)d\ytt,\ \check y=\Yo_x\check x,\ q(\ytt)=\frac{Q\ytt\cdot\ytt}2,\\
&u(\ytt,\e)=\frac{D_1^{-1}Y_0'(0)\ytt}{\e^{1/2}}+\frac{D^{-1}\yor_0}\e
+D_1^{-1}\frac{Y_0''(0)\ytt\cdot\ytt}2.
\end{split}
\ee
Represent $\Mpsi_1$ in terms of its Fourier series:
\be\label{psi1-FT}
\Mpsi_1(\check y,q;u)=\sum_{m\in\mathbb Z^n}\tilde\Mpsi_{1,m}(\check y,q)e^{2\pi i m\cdot u}.
\ee

The columns of $Y_0'(0)$ are vectors that span the tangent space to $\T_{x_0}$ at $y_0=0$ written in the new $y$ coordinates. The columns of $UY_0'(0)$ span the tangent space to $\tilde \T_{x_0}$ at $\yor_0$ written in the original $\yor$ coordinates. By assumption, $\tilde\T_{x_0}$ is generic at $\yor_0$ with respect to $D$ (cf. Definition~\ref{def:gen}), so there is no $m\in\mathbb Z^n$ such that $m\not=0$ and $m D_1^{-1}Y_0'(0)=0$. The same argument as in (5.8)--(5.14) in \cite{kat20a} implies 
\be\label{f1lim}\begin{split}
\lim_{\e\to0}\e^{\kappa} f_\e^{(1)}(x_\e)
= \int_{|\ytt|\le A}&\int_{[0,1]^n}\Mpsi
\biggl(\Yo_x\check x+D_1 u,\frac{Q\ytt\cdot\ytt}2+\bp\cdot D_1 u\biggr)du d\ytt.
\end{split}
\ee
Here is an outline of the argument. Break up the integral with respect to $\ytt$ in \eqref{key-int} into a sum of integrals over a finite, pairwise disjoint covering of the domain of integration by subdomains $B_k$ with diameter $0<\de\ll1$. Then approximate each of these integrals by assuming that $\ytt$ is constant everywhere except in the first term of $u$. This is done by choosing $\ytt_k\in B_k$ in an arbitrary fashion:
\be\label{key-int-rs}\begin{split}
&J(\e)=\sum_k \left[J_k(\e)+O(\de^{a})\text{Vol}(B_k)\right],\
J_k(\e):=\int_{B_k}\Mpsi_1(\check y,q(\ytt_k);u_k(\ytt,\e))d\ytt,\\
&u_k(\ytt,\e)=\frac{D_1^{-1}Y_0'(0)\ytt}{\e^{1/2}}+\left[\frac{D^{-1}\yor_0}\e
+D_1^{-1}\frac{Y_0''(0)\ytt_k\cdot\ytt_k}2\right].
\end{split}
\ee
Thus, the variable of integration $\ytt$ is present only in the rapidly changing term that has $\e^{1/2}$ in the denominator. The magnitude of the error term $O(\de^{a})$ follows from Lemma~\ref{lem:psi-incr}. Represent each $\phi_1$ in \eqref{key-int-rs} in terms of its Fourier series \eqref{psi1-FT}. Using the fact that there is no $m\in\mathbb Z^n$ such that $m\not=0$ and $m D_1^{-1}Y_0'(0)=0$ implies
\be\label{lim-exp}
\lim_{\e\to0}\int_{B_k}\exp(2\pi i m\cdot u(\ytt,\e))d\ytt =0\text{ if } m\not=0.
\ee
By construction, the first term in $u_k$ (cf. \eqref{key-int-rs}) is the only one that contains $\ytt$ and changes rapidly as $\e\to0$. In turn, \eqref{lim-exp} implies 
\be\label{lim}
\lim_{\e\to0} J_k(\e)=\tilde\Mpsi_{1,m=0}(\check y,q(\ytt_k))=\int_{[0,1]^n}\Mpsi_1(\check y,q(\ytt_k);u)du.
\ee
Using \eqref{psi1} and that $\de>0$ can be as small as we like finishes the proof of \eqref{f1lim}.

By \eqref{pPsi-props} and \eqref{f1lim},
\be\label{f1lim-Psi}
\lim_{\e\to0}\e^{\kappa} f_\e^{(1)}(x_\e)
= \int_{|\ytt|\le A}\MPsi
\biggl(\bp\cdot\Yo_x\check x-\frac{Q\ytt\cdot\ytt}2\biggr)d\ytt.
\ee
Since $\pa Y_1/\pa x_\perp=0$ (cf. \eqref{x1y1-ders}), we have
\be\label{thyo}
\bp\cdot\Yo_x\check x = \bp_1\frac{\pa Y_1}{\pa x}\check x=\frac{\pa Y_1}{\pa x_1}\check x_1.
\ee

\section{Estimating the first part of the lower order terms}\label{sec:fplots}
In this section we prove that the first part of the lower order terms does not contribute to the DTB. As usual, by $c$ we denote various positive constants that may have different values in different places. From \eqref{YZ-diff}, \eqref{frst-order-eps}, \eqref{P-expl}, and \eqref{two-sets} it follows that there exists $c_1>0$ such that
\be\label{c1-assn}
\yt\in\Omega_1,\ y\in\T_{x_\e} \text{ implies } |P(y)|\le c_1\e.
\ee
Let $\bt$ and $s$ denote the remaining highest order exponents in \eqref{B-lim} and \eqref{g-lim}, respectively. By construction, $\bt_0-s_0>\bt-s$. This means that either $s=s_0$ if the first term in $\CB$ is missing (i.e., $\bt=\bt_1<\bt_0$), or $\bt=\bt_0$ if the first term in $\fs$ is missing (i.e., $s=s_1>s_0$).

Suppose initially that $\bt > \lfloor s^-\rfloor$. Set $k:=\lceil\bt^+\rceil$, $\nu:=k-\bt$. Thus, $0<\nu\leq 1$, $\nu=1$ if $\bt\in\N_0$, and $s_0\le s<k\le \lceil\bt_0^+\rceil$. Clearly,
\be\label{B-split-lot}
\CB =\CW_1 \pa_{y_1}^k+\CW_2,
\ee
for some $\CW_1\in S^{-\nu}(\vs)$ and $\CW_2\in S^{-\infty}(\vs)$. Here we use a cut-off near $\eta=0$ and the fact that the amplitude of $\CB$ is supported in a small conic neighborhood of $(y_0,\bp=(1,0,\dots,0)^T)$.

In view of \eqref{recon-lead-o}, consider:
\be\label{Psi_dc}
(\CB g_\e)(y)=\sum_{|j|\le O(1/\e)} \left(\CB \ik\left(\frac{\cdot-\vy^j}\e\right)\right)(y)
g(\vy^j).
\ee
Then
\be\label{Psid-lot}
(\CB g_\e)(y)=\int  K(y,y-w) g_\e^{(k)}(w)dw+O(1),
\ee
where $K(y,w)$ is the Schwartz kernel of $\CW_1$, and $O(1)$ represents $\CW_2 g_\e (y)$. The latter statement follows, because $g_\e(y)$ is uniformly bounded as $\e\to0$ for all $y\in\vs$ (cf. \eqref{g-continuity}) and compactly supported. By the estimate (5.13) in \cite{Abels12},
\be\label{ker-est-lot}
|\pa_{w_1}^l K(y,y-w)|\le c(l)|y-w|^{-(n-\nu+l)},\ l\ge 0,\ y,w\in\vs.
\ee
Combining \eqref{B-split-lot}, \eqref{F-def}, the two top cases in \eqref{F-ineq-lot} with $l=k$ and $s_0$ replaced by $s$, \eqref{Psid-lot}, and \eqref{ker-est-lot} with $l=0$, gives
\be\label{J-etwo-lot}\begin{split}
&|(\CB g_\e)(y)|\leq c(J_1+J_2)+O(1),\\
&J_1:=\int_{\varkappa_1\e\leq |P(w)|\leq O(1)} \frac{|w_1-\psi(w_\perp)|^{s-k}}{|y-w|^{n-\nu}} dw,\
J_2:=\int_{|P(w)|\leq \varkappa_1\e} \frac{\e^{s-k}}{|y-w|^{n-\nu}} dw,
\end{split}
\ee
where $\varkappa_1$ is the same as in Lemma \ref{lem:gel}. Consider $J_1$:
\be\label{J1-lot}
J_1=\int\int_{\varkappa_1\e\leq |p|\leq O(1)} \frac{|p|^{s-k}}{|([P-p]+\psi(y_\perp)-\psi(w_\perp),y_\perp-w_\perp)|^{n-\nu}} dp dw_\perp,
\ee
where we denoted $P:=P(y)$ and changed variables $w_1\to p=w_1-\psi(w_\perp)$. There exists $0<c'<1$ so that
\be\label{denom-est}
|a+\psi(y_\perp)-\psi(w_\perp)|+|y_\perp-w_\perp|\geq c'(|a|+|y_\perp-w_\perp|),\ a\in\br,y,w\in\vs.
\ee
By construction, $\psi'(0)=0$. Assume $\vs$ is sufficiently small, so that $|\psi(y_\perp)-\psi(w_\perp)|\leq c''|y_\perp-w_\perp|$, $y,w\in\vs$, for some $0<c''<1$. Then any $c'$ such that $0<c'<1-c''$ works. This implies
\be\label{J1-ethree-lot}\begin{split}
J_1\leq &c\int_{\varkappa_1\e\leq |p|\leq O(1)}\int \frac{|p|^{s-k}}{(|P-p|+|w_\perp|)^{n-\nu}} dw_\perp dp\\
\leq & c\int_{\varkappa_1\e\leq |p|\leq O(1)} \frac{|p|^{s-k}}{|P-p|^{1-\nu}} dp=\begin{cases} O(\e^{s-\bt}),& \bt>s,\\ O(\ln(1/\e)),&\bt=s,\\ O(1),&\bt<s.\end{cases}
\end{split}
\ee
Here we have used that $P=O(\e)$.

The term $J_2$ can be estimated analogously, and we get an estimate similar to \eqref{J1-ethree-lot}, where the bound is $O(\e^{s-\bt})$ in all three cases. 

Suppose $\bt>s$. By \eqref{J-etwo-lot}, $(\CB g_\e)(y)=O(\e^{s-\bt})$. Estimate the integral in \eqref{recon-lead-o}:
\be\label{f1-est-lot}
\begin{split}
|\e^\kappa f_\e^{(1)}(x_\e)|\le & \e^\kappa \int_{\Omega_1} O(\e^{s-\bt}) d\yt=
O(\e^{\kappa+s-\bt})\int_0^{A\e^{1/2}}r^{N-1}dr\\
=&O(\e^{(\bt_0-s_0)-(\bt-s)})\text { if }\bt>s.
\end{split}
\ee
In a similar fashion,
\be\label{f1-est-lot-extra}
|\e^\kappa f_\e^{(1)}(x_\e)|=\begin{cases}
O(\e^{\bt_0-s_0}\ln(1/\e)), & \bt=s,\\
O(\e^{\bt_0-s_0}), & \bt<s.
\end{cases}
\ee
Since $\bt_0-s_0\ge N/2\ge1/2$, $\e^\kappa f_\e^{(1)}(x_\e)\to0$ in all three cases.

Suppose now $0<\bt \le \lfloor s^- \rfloor$. Similarly to \eqref{B-split-lot}, 
$\CB =\CW_1 \pa_{y_1}^k+\CW_2$, where $k=\lceil \bt \rceil\ge1$, $\nu=k-\bt\ge 0$, $s>k$, $\CW_1\in S^{-\nu}(\vs)$, and $\CW_2\in S^{-\infty}(\vs)$. The kernel of $\CW_1$ is an $L^1$ function (see e.g. Theorem 5.15 in \cite{Abels12}) and $\sup_{y\in\vs}|\Delta g_\e^{(k)}(y)|=O(\e^{s-k})$ (cf. \eqref{F-ineq-2} with $s_0$ replaced by $s$). This implies that  $\sup_{y\in\vs}|(\CB g_\e)(y)-\CB g(y)|=O(\e^{s-k})$. From Lemma~\ref{lem:conorm}, $\CB g\in C_*^{s-\bt}(\vs)$, and $s>\bt$. Thus, $(\CB g_\e)(y)=O(1)$, and the desired result follows similarly to the case $\bt<s$ in \eqref{f1-est-lot-extra}. The case $\bt\le0$ is proven using the same argument with $l=0$ in \eqref{F-ineq-2} and without splitting $\CB$ into two parts.

\section{Estimating the second part of the DTB}\label{sec:spdtb}
Next, consider the quantity $f_\e^{(2)}(x_\e)$:
\be\label{recon-lead-t}\begin{split}
f_\e^{(2)}(x_\e)= &\int_{\substack{y\in\T_{x_\e}\\ |\yt|> A\e^{1/2}}}\sum_{|j|\le O(1/\e)} \left(\CB \ik\left(\frac{\cdot-\vy^j}\e\right)\right)(y)
g(\vy^j) w(x_\e,y)dy,
\end{split}
\ee
where both $\CB$ and $g$ are given by their full expressions. The continuous counterpart of \eqref{recon-lead-t} is 
\be\label{recon-lead-cont-t}
f^{(2)}(x_\e)= \int_{\substack{y\in\T_{x_\e}\\ |\yt|> A\e^{1/2}}} (\CB g)(y)
w(x_\e,y)dy.
\ee


The following lemma is proven in Appendix~\ref{sec:pPdiff}.

\begin{lemma}\label{lem:pPdiff} Suppose $\CB$, $g$, and $\ik$ satisfy the assumptions in Section~\ref{sec:assns}. There exist $c,\varkappa_2>0$ such that for all $\e>0$ sufficiently small one has
\be\label{bt-aux-est}
|(\CB g_\e)(y)-(\CB g)(y)|\le c\e \left|P(y)\right|^{s_0-1-\bt_0}\begin{cases}1,&\bt_0\not\in\N,\\
|\ln(P(y)/\e)|,&\bt_0\in\N,\end{cases}\ y\in\vs,
\ee
whenever $|P(y)|>\varkappa_2\e$.
\end{lemma}

Return now to \eqref{recon-lead-t}. 
Pick any $y\in\T_{x_\e}$. Recall that $y=Y(\yt,x)$ is obtained by solving $x=\Phi(t,y)$ for $t$ and $\yo$, and that $\yt\equiv Y^{\tp}(\yt,x)$. Since $\det \Phi_{(t,\yo)}\not=0$, $Y(\yt,x)$ is a smooth function of $x$. Hence $|Y(\yt,x_\e)-Y_0(\yt)|=O(\e)$. Strictly speaking, we cannot invoke \eqref{frst-order-eps} here, because in \eqref{frst-order-eps} the assumption is $|\yt|=O(\e^{1/2})$. Therefore,
\be\label{P-alt}
P(Y(\yt,x_\e))=P(Y_0(\yt))+O(\e)=\bp\cdot(Y_0(\yt)-Z(\yt))+O(\e).
\ee
Recall that $Z(\yt)$ is the projection of $Y_0(\yt)$ onto $\Gamma$, cf. Lemma~\ref{lem:tang}.
Using \eqref{YZ-diff}, \eqref{two-sets}, and that $Q$ is negative definite, by shrinking $\vs$, if necessary, and taking $A\gg1$ large enough, we can make sure that (a) $P(y)\ge c|\yt|^2$ for some $c>0$ and (b) inequality \eqref{bt-aux-est} applies (i.e. $P(y)>\varkappa_2\e$) if $y\in\T_{x_\e}$ and $\yt\in\Omega_2$ for all $\e>0$ small enough.

Suppose first that $\kappa>0$. Using \eqref{Psi_dc}, \eqref{recon-lead-t}, \eqref{bt-aux-est}, and \eqref{Psic-ass-Bg} gives an estimate
\be\label{f2-est}\begin{split}
|\e^\kappa f_\e^{(2)}(x_\e)|\le &O(\e^\kappa)\int_{\Omega_2}\left(\e P(y)^{s_0-\bt_0-1}\ln(P(y)/\e)+P(y)^{s_0-\bt_0}\right)d\yt\\
\le & O(\e^\kappa)\int_{\Omega_2}\left(\e |\yt|^{2(s_0-\bt_0-1)}\ln(P(y)/\e)+|\yt|^{2(s_0-\bt_0)}\right)d\yt\\
=&O(A^{-2\kappa}).
\end{split}
\ee
If $\kappa=0$, we get from \eqref{recon-lead-cont-t}, \eqref{bt-aux-est}, and \eqref{Psic-ass-extra}
\be\label{f2-est-k0}\begin{split}
|f_\e^{(2)}(x_\e)-f^{(2)}(x_\e)|\le &\int_{\Omega_2}\e P(y)^{s_0-\bt_0-1}\ln(P(y)/\e)d\yt\\
\le & \e\int_{\Omega_2}|\yt|^{2(s_0-\bt_0-1)}\ln(P(y)/\e)d\yt
=O(A^{-2}\ln A).
\end{split}
\ee
As $A\gg1$ can be arbitrarily large, combining \eqref{f2-est} and \eqref{f2-est-k0} with \eqref{f1lim-Psi}, \eqref{thyo} yields
\be\label{f1lim-1}
\lim_{\e\to0}\begin{cases} \e^{\kappa}  \check f_\e(x_\e),&\kappa>0\\ \check f_\e(x_\e)-f^{(2)}(x_\e),&\kappa=0 \end{cases}
= \int_{\br^{N}}\MPsi \biggl(\frac{\pa Y_1}{\pa x_1}\check x_1-\frac{Q\ytt\cdot\ytt}2\biggr)d\ytt.
\ee

Assume again that $\kappa=0$.
By \eqref{Psic-ass-extra} and that $P(y)>c|\yt|^2$ if $y\in\T_{x_\e}$ and $\yt\in\Omega_2$, it follows that the integral in \eqref{recon-lead-cont-t} admits a uniform (i.e., independent of $\e>0$ sufficiently small, $A\gg1$ sufficiently large, and $\check x$ confined to a bounded set) integrable bound:
\be\label{f2-bound}\begin{split}
\int_{|\yt|\le O(1)} P(y)^{s_0-\bt_0+c}d\yt
\le & c'\int_{|\yt|\le O(1)}|\yt|^{2(s_0-\bt_0+c)}d\yt\\
\le & c'\int_0^{O(1)}r^{2(s_0-\bt_0+c)}r^{N-1}dr<\infty,\ \bt_0-s_0=N/2.
\end{split}
\ee
In \eqref{f2-bound}, the constant $c$ in the exponent is the same as the one in \eqref{Psic-ass-extra}. Therefore, we can compute the limit of $f^{(2)}(x_\e)$ as $\e\to0$ by taking the pointwise limit of the integrand in \eqref{recon-lead-cont-t}. This limit is independent of $A\gg1$. Shrinking $\vs$ if necessary, by \eqref{YZ-diff} we can ensure that $P(y)>0$ for any $y\in\T_{x_0}$, $y\not=0$. Hence
\be\label{f2c-lim}
\lim_{\e\to0}f^{(2)}(x_\e) = \int_{\T_{x_0}} (\CB g)(y)w(x_0,y)dy:=\lim_{\de\to0}\int_{\substack{y\in\T_{x_0}\\ |\yt|> \de}} (\CB g)(y)w(x_0,y)dy.
\ee
Thus, the limit is independent of $\check x$. 

A slightly more general argument holds as well. Let $x=(x_1,0,\dots 0)^T\in \us$ be a point with $x_1>0$ sufficiently small, and let $y\in\T_x$ be arbitrary. It follows from \eqref{YZ-diff} and $\pa Y_1/\pa x_1>0$ (see \eqref{x1y1-ders}) that $P(Y(\yt,x))\ge c(x_1+|\yt|^2)$ for some $c>0$. For any such $x$, we still have the same lower bound $P(Y(\yt,x))\ge c|\yt|^2$. Adopt the convention that the interior side of $\s$ is the one where the $x_1$ axis points and define
\be\label{int-ext}
x_0^{\text{int}}:=\lim_{x_1\to 0^+}(x_1,x_\perp=0),\quad x_0^{\text{ext}}:=\lim_{x_1\to 0^-}(x_1,x_\perp=0).
\ee
In view of \eqref{f2-bound}, we can use dominated convergence to conclude
\be\label{f2c-lim-alt}
\int_{\T_{x_0}} (\CB g)(y)w(x_0,y)dy=\lim_{\substack{x=(x_1,x_\perp=0)\\x_1\to 0^+}}\int_{y\in\T_x} (\CB g)(y)w(x_0,y)dy=(\R^*\CB g)(x_0^{\text{int}}).
\ee

\section{Computing the DTB}\label{sec:compdtb}

The right side of \eqref{f1lim-1} simplifies to the expression
\be\label{f1lim-2}
\frac{2^{N/2}}{|\det Q|^{1/2}}
\int_{\br^N}\MPsi\left(\frac{\pa Y_1}{\pa x_1}\check x_1+\vert v \vert^2\right)dv
=\frac{2^{N/2}|S^{N -1}|}{|\det Q|^{1/2}}
\ioi\MPsi\left(\frac{\pa Y_1}{\pa x_1}\check x_1+q^2\right)q^{N-1}dq.
\ee
Set (see \eqref{Psi_a})
\be\label{J-def}\begin{split}
J(h):=&\frac12 \int\MPsi(h+q)q_+^{(N -2)/2}dq
=\frac{e(-N/2)\Gamma(N/2)}2 \CF_{1d}^{-1}(\tilde\MPsi(\la)(\la-i0)^{-N/2}),\\ 
h=&\frac{\pa Y_1}{\pa x_1}\check x_1.
\end{split}
\ee
Using \eqref{recon-lead-o1}, \eqref{oper1d}, \eqref{pPsi-def}, and \eqref{g-FT} gives
\be\label{J-ft-v2}\begin{split}
J(h)=&\frac{\Gamma(N/2)w(x_0,y_0)}2 \CF_{1d}^{-1}\left(\tilde\ik(\la\bp)\mu(\la)\right)(h),\\
\mu(\la):=&\tilde B_0(y_0,\bp)\fs^+(y_0)e(-N/2)\la_+^{\kappa-1}+\tilde B_0(y_0,-\bp)\fs^-(y_0)e(N/2)\la_-^{\kappa-1}.
\end{split}
\ee
Here we have used that $\pa\Yo/\pa\yt=0$ (see at the end of the proof of Lemma~\ref{lem:tang_0}) and $Y^\tp(\yt,x)\equiv\yt$, so $G^{\T}$ is the identity matrix, and $\det G^{\T}=1$. 

The function $J(h)$ is identical to the one introduced in (4.6) of \cite{kat20b} if we replace $n$ in the latter with $N+1$. See Section 4.5 of \cite{kat20b} for additional information about this function. In particular, $\Upsilon(p)=O(|p|^{-(\bt_0-s_0)})$, $p\to\infty$ (this follows from \eqref{pPsi-def}), hence the integral in \eqref{J-def} is absolutely convergent if $\kappa>0$. Also, $\mu(\la)$ is the product of three distributions $\tilde b(\la)\tilde\CA(\la)(\la-i0)^{-N/2}$, which is well-defined as a locally integrable function if $\kappa>0$.

Combine with the constant in \eqref{f1lim-2} and compute the inverse Fourier transform (cf. \eqref{Psi_a})
\be\begin{split}\label{DTB-v1}
&DTB(h)=C_1\int \hat\ik(\bp,h-p)\left(c_1^+(p-i0)^{-\kappa}+c_1^-(p+i0)^{-\kappa}\right)dp,\ \kappa>0,\\ 
&C_1=(2\pi)^{N/2}w(x_0,y_0)|\det Q|^{-1/2}, \
c_1^\pm=\frac{\Gamma(\kappa)}{2\pi}\tilde B_0(y_0,\pm\bp)\fs^\pm(y_0)e(\mp (\bt_0-s_0)).
\end{split}
\ee

If $\kappa=0$, a more careful analysis of $J(h)$ is required (see Section 4.6 of \cite{kat20b}). Similarly to \eqref{Bvsum-k0}, \eqref{Bg-k0}, condition \eqref{c-ratios-req} implies that $\Upsilon(p)\equiv 0$, $p>c$, for some $c>0$, hence the integral in \eqref{J-def} is still absolutely convergent. Straightforward multiplication of the distributions to obtain $\mu(\la)$ no longer works, because $\mu(\la)$ is not a locally integrable function if $\kappa=0$. Fortunately, in this case $\mu(\la)$ is computed in (4.45) of \cite{kat20b} (see (4.45)--(4.47) in \cite{kat20b}). Observe that condition \eqref{c-ratios-req} in this paper is equivalent to condition (4.6) of \cite{kat20b}. At first glance the two conditions differ by a sign, but $s_0$ in \cite{kat20b} corresponds to $s_0+1$ here, which eliminates the discrepancy. Then $\mu(\la)=\tilde B_0(y_0,\bp)\fs^+(y_0)e(-N/2)(\la-i0)^{-1}$, and \eqref{DTB-v1} becomes
\be\begin{split}\label{DTB-v2}
DTB(h)=&C_1c_1\int \hat\ik(\bp,h-p)p_-^0 dp=C_1c_1\int_{-\infty}^0 \hat\ik(\bp,h-p)dp,\ \kappa=0,\\
c_1:=&i\tilde B_0(y_0,\bp)\fs^+(y_0) e(-(\bt_0-s_0)),
\end{split}
\ee
where $C_1$ is the same as in \eqref{DTB-v1}.  

Let us now compute $|\det Q|^{1/2}$. This can be done by eliminating the auxiliary variable $t$ (see Section~\ref{sec:prelims}). Since $\det \Phi^\tp_t\not=0$, we can solve $\xt=\Phi^\tp(t,y)$ for $t$. This gives a smooth function $t=T(\xt,y)$. Then we can define a function $X(\xt,y):=\Phi(T(\xt,y),y)$, which parameterizes the surfaces $\s_y$ in terms of $\xt$. By construction, this function satisfies
\be\label{X-ids} 
X^\tp(\xt,y)\equiv\xt \text{ and } X(\Phi^\tp(t,y),y)\equiv \Phi(t,y).
\ee
The following lemma is proven in Appendix~\ref{sec:dets}.
\begin{lemma}\label{lem:detQ} One has
\be\label{Q-final}
|\det Q|^{1/2}=\left|({\pa X_1}/{\pa y_1})^N{\det\dsf}\right|^{-1/2}\left|\det\frac{\pa^2 X_1}{\pa \xt\pa\yt}\right|
\ee
and
\be\label{x1y1-ders}
{\pa Y_1}/{\pa x_1}=({\pa X_1}/{\pa y_1})^{-1},\quad {\pa Y_1}/{\pa x_1},{\pa X_1}/{\pa y_1}>0,\quad \pa Y_1/\pa x_\perp=0,
\ee
where $y=Y(\yt,x)$ is the function defined in the paragraph preceding \eqref{P-alt}.
\end{lemma}
Recall that $\pa/\pa\xt_j$, $1\le j\le N$, are basis vectors spanning the tangent space to $\s_{y_0}$ at $x_0$. Likewise, $\pa/\pa\yt_j$, $1\le j\le N$, are basis vectors spanning the tangent space to $\T_{x_0}$ at $y_0$.

Consider now a particular case, where $g=\R f$, and $f$ is given by \eqref{f-orig}, \eqref{f-lim}. Then $\R f$ is given by \eqref{GRT-f-alt}, \eqref{f-lim-coefs}. From Assumption C3 and \eqref{jacob-sm}, $\text{sgn}( (\Psi\circ\Phi)_{tt}(t^*,y))=-(n-1)$. The analogues of \eqref{DTB-v1}, \eqref{DTB-v2} are computed to be 
\be\begin{split}\label{DTB-v3}
DTB(h)=&C_2\int \hat\ik(\bp,h-p)\left(c_2^+(p-i0)^{-\kappa}+c_2^-(p+i0)^{-\kappa}\right)dp,\ \kappa>0,\\ 
C_2=&(2\pi)^N b(x_0,y_0)w(x_0,y_0)\left|\frac{\det G^{\s}}{\det Q\det (\Psi\circ\Phi)_{tt}}\right|^{1/2}, \\
c_2^\pm=&\frac{\Gamma(\kappa)}{2\pi}\tilde B_0(y_0,\pm\bp)\tilde f^\pm(x_0)e\left(\mp\bigl( \bt_0-s_0-\frac{n-1}2\bigr)\right),
\end{split}
\ee
and 
\be\begin{split}\label{DTB-v4}
DTB(h)=&C_2c_2\int \hat\ik(\bp,h-p)p_-^0 dp,\ \kappa=0,\\
c_2:=&i\tilde B_0(y_0,\bp)\tilde f^+(x_0)e\left(-\bigl( \bt_0-s_0-\frac{n-1}2\bigr)\right).
\end{split}
\ee

The following lemma is proven in Appendix~\ref{sec:dets}.
\begin{lemma}\label{lem:chi} One has
\be\label{dets-only-v2}\begin{split}
\chi:=\left|\frac{\det G^{\s}}{\det Q\det (\Psi\circ\Phi)_{tt}}\right|^{1/2}
=\left(\frac{\pa X_1}{\pa y_1}\right)^N\left|{\det\frac{\pa^2 X_1}{\pa \xt\pa\yt}}\right|^{-1}.
\end{split}
\ee
\end{lemma}

Now we can state our main result.
\begin{theorem} Suppose 
\begin{enumerate}
\item $\CB$, $g$, and $\ik$ satisfy the assumptions in Section~\ref{sec:assns}, 
\item Conditions G1--G3 in Section~\ref{sec:prelims} are satisfied, 
\item The pair $(x_0,\yor_0)$, $\tilde y_0\in \tilde \T_{x_0}$, is generic for the sampling matrix $D$, and
\item There is $\xi_0\in T_{x_0}^*\us$, $\xi_0\not=0$, such that $\xi_0\Phi_{(t,y)}=(0,\Theta)$, where $\Theta$ is conormal to $\Gamma:=\text{sing\hspace{1.0pt}supp}(g)$ at $y_0$.
\end{enumerate}
Then one has
\be\label{DTB-nz}\begin{split}
\lim_{\e\to0} \e^{\kappa}  \check f_\e(x_\e)
= &C_1\int \hat\ik\left(\bp,\left(\frac{\pa X_1}{\pa y_1}\right)^{-1}\check x_1-p\right)\left(c_1^+(p-i0)^{-\kappa}+c_1^-(p+i0)^{-\kappa}\right)dp, \\ 
C_1=&(2\pi)^{N/2}w(x_0,y_0)|\det\dsf|^{1/2}\left(\frac{\pa X_1}{\pa y_1}\right)^{N/2}\left|\det\frac{\pa^2 X_1}{\pa \xt\pa\yt}\right|^{-1}, \\
c_1^\pm=&\frac{\Gamma(\kappa)}{2\pi}\tilde B_0(y_0,\bp)\fs^+(y_0)e(\mp (\bt_0-s_0)),\ \kappa>0,
\end{split}
\ee
and
\be\label{DTB-z}\begin{split}
\lim_{\e\to0} \check f_\e(x_\e)=&\check f(x_0^{\text{int}})+C_1c_1\int_{-\infty}^0 \hat\ik\left(\bp,\left(\frac{\pa X_1}{\pa y_1}\right)^{-1}\check x_1-p\right)dp,\\
c_1:=&i\tilde B_0(y_0,\bp)\fs^+(y_0) e(-(\bt_0-s_0)),\ \kappa=0.
\end{split}
\ee
Suppose, in addition, that $g=\R f$, and $f$ is given by \eqref{f-orig}, \eqref{f-lim}. Then 
\be\label{DTB-nz-v2}\begin{split}
\lim_{\e\to0} \e^{\kappa}  \check f_\e(x_\e)
= &C_2\int \hat\ik\left(\bp,\left(\frac{\pa X_1}{\pa y_1}\right)^{-1}\check x_1-p\right)\left(c_2^+(p-i0)^{-\kappa}+c_2^-(p+i0)^{-\kappa}\right)dp, \\ 
C_2=&(2\pi)^N b(x_0,y_0)w(x_0,y_0)\left(\frac{\pa X_1}{\pa y_1}\right)^N\left|{\det\frac{\pa^2 X_1}{\pa \xt\pa\yt}}\right|^{-1}, \\
c_2^\pm=&\frac{\Gamma(\kappa)}{2\pi}\tilde B_0(y_0,\pm\bp)\tilde f^\pm(x_0)e\left(\mp\bigl( \bt_0-s_0-\frac{n-1}2\bigr)\right),\ \kappa>0,
\end{split}
\ee
and 
\be\label{DTB-z-v2}\begin{split}
\lim_{\e\to0} \check f_\e(x_\e)=&\check f(x_0^{\text{int}})
+C_2c_2\int_{-\infty}^0 \hat\ik\left(\bp,\left(\frac{\pa X_1}{\pa y_1}\right)^{-1}\check x_1-p\right) dp,\\
c_2:=&i\tilde B_0(y_0,\bp)\tilde f^+(x_0)e\left(-\bigl( \bt_0-s_0-\frac{n-1}2\bigr)\right),\ \kappa=0.
\end{split}
\ee
\end{theorem}
Recall that $\check f(x):=\R^*\CB g(x)$ denotes the exact reconstruction from continuous data. 
The second term on the right in \eqref{DTB-z} equals zero for all $\check x_1>c$. Because $\ik$ is normalized: $\int\hat\ik(\bp,p)dp=1$, the second term equals $C_1c_1$ for all $\check x_1<-c$. Here $c>0$ is sufficiently large, and we used that $\pa X_1/\pa y_1>0$ (cf. \eqref{x1y1-ders}). By \eqref{CTB-st3}, the product $C_1c_1$ is precisely the jump of the exact reconstruction $\check f(x)$ across $\s$ at $x_0$: $C_1c_1=\check f(x_0^{\text{ext}})-\check f(x_0^{\text{int}})$, see \eqref{f2c-lim-alt} and \eqref{int-ext}. Thus, the right-hand side of \eqref{DTB-z} equals to $\check f(x_0^{\text{int}})$ if $\check x_1>c$, and to $\check f(x_0^{\text{ext}})$ -- if $\check x_1<-c$. 
This shows that \eqref{DTB-z} describes a smooth transition of the discrete reconstruction  $\check f_\e(x_\e)$ from the value $\check f(x_0^{\text{int}})$ on the interior side of $\s$ to the value $\check f(x_0^{\text{ext}})$ on the exterior side of $\s$. Loosely speaking, the transition happens over a region of size $O(\e)$:
\be\label{transition}\begin{split}
\text{DTB}(\check x)=\lim_{\e\to0} \check f_\e(x_0+\e \check x)=\begin{cases} \check f(x_0^{\text{int}}), &\check x_1>c,\\
\check f(x_0^{\text{ext}}),& \check x_1<-c.
\end{cases}
\end{split}
\ee
Then the DTB is a ``stretched'' version of the abrupt jump of $\check f$ across $\s$ in the continuous case. The DTB expression in \eqref{DTB-z-v2} is a particular case of the one in \eqref{DTB-z}, so the same intuition applies to the former. See also Section 6 of \cite{kat20a} for a similar discussion in the setting of quasi-exact inversion of the GRT in $\br^3$.


\section{Computing the CTB}\label{sec:compctb}

When describing the leading singularity of a distribution at a point, the following definition (which is a slight modification of the one in \cite{kat99_a}) is convenient.

\begin{definition}[\cite{kat99_a}]\label{lead-sing}
Given a distribution $f\in{\mathcal D}'(\br^n)$ and a point $x_0\in\br^n$, suppose there exists a distribution $f_0\in{\mathcal D}'(\br^n)$ so that for some $a\in\br$ the following equality holds 
\be\label{losing}\begin{split}
\lim_{\e\to0} \e^{-a}\int f(x_0+\e\check x)\pa_{\check x}^m\om(\check x)d\check x=&\int f_0(\check x)\pa_{\check x}^m\om(\check x)d\check x,\\ 
\forall m\in\N_0^n,\ |m|=&\max(0,\lceil a^+\rceil),
\end{split}
\ee
for any $\om\in\coi(\br^n)$. Then we call $f_0$ the leading order singularity of $f$ at $x_0$.
\end{definition}

From \eqref{Psi_c-expr}--\eqref{Psi_c-f-FT} and \eqref{recon} it follows that
\be\label{CTB-st1}\begin{split}
&(\R^*\CB g)(x)=\frac1{2\pi}\int_\br \int_{\br^N} J(y,\la) e^{-i\la P(y)}w(x,y)(\det G^{\T}(x,\yt))^{1/2}d\yt d\la, \\
&J(y,\la)=\tilde B_0(y,P'(y))\fs^+(y)\la_+^{\bt_0-s_0-1}+\tilde B_0(y,-P'(y))\fs^-(y)\la_-^{\bt_0-s_0-1}+\tilde R(y,\la), \\
&\tilde R\in S^{c-1}(\vs\times\br),\ y=Y(\yt,x)\in\vs,\ 
c=\max(\bt_0-s_0-1,\bt_1-s_0,\bt_0-s_1).
\end{split}
\ee

By construction, $P(Y_0(\yt))=\bp\cdot(Y_0(\yt)-Z(\yt))$. By \eqref{YZ-diff}, 
\be
\pa_{\yt}P(Y_0(\yt))|_{\yt=0}=0,
\ee
the Hessian of $P(Y_0(\yt))$ is non-degenerate at $\yt=0$, and
\be\label{Hess-phase}
\frac{\pa^2 P(Y_0(\yt))}{(\pa\yt)^2}=-Q.
\ee
Therefore the stationary point $\yt_*(x)$ of the phase $P(Y(\yt,x))$ is a smooth function of $x$ in a neighborhood of $x=x_0$. Set $x=x_\e=x_0+\e\check x$. 

By \eqref{frst-order-eps}, \eqref{PM11}, and \eqref{Hess-phase}, application of the stationary phase method to the integral with respect to $\yt$ in \eqref{CTB-st1} implies: 
\be\label{limf-inner-expl}\begin{split}
W_\e(\check x,\la)&:=\int_{\br^N} J(y,\la) e^{-i\la P(y)}w(x_\e,y)(\det G^{\T}(x_\e,\yt))^{1/2}d\yt\\
&=\biggl[(2\pi)^{\frac N2}\left(\frac{w(x_0,y_0)}{|\det Q|^{1/2}}+O(\e)\right)\biggl((\tilde B_0(y_0,\bp)\fs^+(y_0)+O(\e))e\left(-\frac N2\right)\la_+^{\kappa-1}\\
&\hspace{2cm}+(\tilde B_0(y_0,-\bp)\fs^-(y_0)+O(\e))e\left(\frac N2\right)\la_-^{\kappa-1}\biggr)+\tilde R_\e(\check x,\la)\biggr]\\
&\quad \times e^{-i\la \e (h+O(\e))},\quad
\tilde R_\e\in S^{\kappa-1-c}(U\times\br),\ h=\frac{\pa Y_1}{\pa x_1}\check x_1,\ \kappa>0,
\end{split}
\ee 
for some $c>0$ and any open, bounded set $U\subset\br^n$. The expression $h+O(\e)$ in the exponent arises due to \eqref{frst-order-eps}, \eqref{PM11}, \eqref{thyo}, and because $P(y)=y_1-\psi(y_\perp)$, where $\psi(0)=0$ and $\psi'(0)=0$. The constants $c_{m,\al}$ that control the derivatives of $\tilde R_\e$ in \eqref{symbol-ineqs} can be selected independently of $\e$ for all $\e>0$ sufficiently small. Likewise, it follows from \eqref{J-bnd} that for any open, bounded $U\subset \br^n$ and $c'>0$ there exists $c_m(U,c')$ such that
\be\label{We-bnd}
|\pa_{\check x}^m W_\e(\check x,\la)|\le c_m(U,c'),\ (x,\la)\in U\times[0,c'],m\in\N_0^n, 
\ee
for all $\e>0$ sufficiently small.

The analogue of \eqref{Psi_c-f-FT}, \eqref{Psi_c-f} becomes
\be\label{CTB-st2}\begin{split}
&\lim_{\e\to0}\e^{\kappa}(\R^*\CB g)(x_\e)
= C_1\CF_{1d}^{-1}(\mu(\la))(h)=C_1\left(c_1^+(h-i0)^{-\kappa}+c_1^-(h+i0)^{-\kappa}\right),\ \kappa>0,
\end{split}
\ee
where $\mu(\la)$ is the same as in \eqref{J-ft-v2}, and $C_1$ and $c_1^\pm$ are the same as in \eqref{DTB-v1}. The limit in \eqref{CTB-st2} is understood in the sense of distributions with test functions $\om\in\coi(\br^n)$ (cf. \cite{kat99_a} and \eqref{losing}). 

When $a=-\kappa<0$, $m=0$, and it is not necessary to require derivatives in \eqref{losing}. Thus, the limit in \eqref{CTB-st2} is understood in the sense of taking the limit as $\e\to0$ on both sides of the following equality:
\be\label{limf-W}\begin{split}
\int_{\br^n} \e^{\kappa}(\R^*\CB g)(x_0+\e\check x)\om(\check x)d\check x=
\frac1{2\pi}\int_\br\int_{\br^n} \e^{\kappa-1}W_\e(\check x,\sigma/\e)\om(\check x)d\check x d\sigma.
\end{split}
\ee 
The right-hand side of \eqref{CTB-st2} follows from \eqref{limf-inner-expl}, the dominated convergence theorem, and \eqref{Psi_a}. The dominated convergence theorem can be applied because 
\begin{enumerate}
\item The first term inside the brackets in \eqref{limf-inner-expl} is absolutely integrable at $\la=0$ since $\kappa>0$, and 
\item The integrand is rapidly decreasing as $\la\to\infty$. This follows from integration by parts with respect to $\check x_1$ on the right in \eqref{limf-W} and using \eqref{We-bnd},  $\om\in\coi(\br^n)$, and ${\pa Y_1}/{\pa x_1}\not=0$ (cf. \eqref{x1y1-ders}).
\end{enumerate}

If $\kappa=0$, the function $\mu(\la)$ is no longer integrable at the origin, and $|m|=1$ in \eqref{losing}. Hence we use test functions of the form $\pa_{\check x_j} \om(\check x)$, $1\le j\le n$. This makes the same argument as in the case $\kappa>0$ to work, but the price to pay is that the CTB is determined up to a constant. See the paragraph following Theorem 4.6 of \cite{kat20b} for a similar phenomenon. Using condition \eqref{c-ratios-req} in \eqref{limf-inner-expl} implies 
\be\label{CTB-st3}\begin{split}
&\lim_{\e\to0}(\R^*\CB g)(x_\e)-C_1c_1(-1/2)\text{sgn}(\check x_1)=\text{const},\ \kappa=0,
\end{split}
\ee
where $c_1$ is the same as in \eqref{DTB-v2}. 

Comparing \eqref{DTB-v1} and \eqref{DTB-v2} with \eqref{CTB-st2} and \eqref{CTB-st3}, respectively, we see that the DTB is the convolution of the CTB with the scaled classical Radon transform of the interpolating kernel. The difference between $p_-^0$ in \eqref{DTB-v2} and $(-1/2)\text{sgn}(p)$ in \eqref{CTB-st3} is due to the nonuniqueness (up to a constant).

\appendix

\section{Proofs of lemmas \ref{lem:conorm}--\ref{lem:delta}}\label{sec:prfPc}

\subsection{Proof of Lemma \ref{lem:conorm}}\label{subsec:conorm}
The expression for $g$ is obtained directly from \eqref{g-def}, \eqref{g-lim}:
\be\label{g-FT}\begin{split}
g(y)=&G(y,P(y)),\\
G(y,p):=&\CF_{1d}^{-1}\left(\fs^+(y)\la_+^{-(s_0+1)}+\fs^-(y)\la_-^{-(s_0+1)}+\tilde R(y,\la)\right)(p),
\end{split}
\ee
where $\tilde R\in S^{-(s_1+1)}(\vs\times\br)$, and $\CF_{1d}^{-1}$ is the one-dimensional inverse Fourier transform acting with respect to $\la$. By the properties of $\tilde R$, we get by computing the inverse Fourier transform
\be\label{g-Psi_nN}
G(y,p) = \fs^+(y)\Psi_{-s_0}^+(p)+\fs^-(y) \Psi_{-s_0}^-(p)+R(y,p).
\ee
Here (see \cite{gs}, p. 360)
\be\label{Psi_a}
\Psi_a^{\pm}(p)=\CF_{1d}^{-1}(\la_{\pm}^{a-1})(p)=\frac{\Gamma(a)}{2\pi}e(\mp a)(p\mp i0)^{-a},\ a\not=0,-1,-2,\dots,
\ee
and $R(p,y)$ is the remainder. By Theorem 5.12 in \cite{Abels12}, $R$ satisfies
\be\label{g-rmd}
|\pa_y^m\pa_p^l R(y,p)|\le c_{m,l}\begin{cases} |p|^{s_1-l},& s_1<l,\\
1+|\log|p||,& s_1=l,\\ 1,& s_1>l,\end{cases}\ m\in\N_0^n,l\in\N_0,y\in\vs,p\not=0,
\ee
for some $c_{m,l}>0$.
Combining \eqref{g-FT}--\eqref{g-rmd} gives the leading singular behavior of $g$: 
\be\label{g-ls-v2}\begin{split}
g(y)\sim & a^+(\bar y)P_+^{\gao}(y)+a^-(\bar y)P_-^{\gao}(y),\\
a^\pm(\bar y)= & \frac{\Gamma(-s_0)}{2\pi}\left(\fs^+(\bar y)e(\pm s_0)+\fs^-(\bar y)e(\mp s_0)\right),\ \bar y:=(\psi(y_\perp),y_\perp)\in\Gamma,\ s_0\not\in\N.
\end{split}
\ee
If $s_0\in\N$, condition \eqref{s-except} implies $\fs^+(\bar y)\la_+^{-(s_0+1)}+\fs^-(\bar y)\la_-^{-(s_0+1)}\equiv \fs^+(\bar y)\la^{-(s_0+1)}$, so (see \cite{gs}, p. 360)
\be\label{g-Psi_N}\begin{split}
g(y)=&\fs^+(\bar y)\Psi_{-s_0}(P(y))+R(P(y),y),\\ 
\Psi_{-s_0}(p)=&\CF_{1d}^{-1}(\la^{-(s_0+1)})(p)=\frac12\frac{(-i)^{s_0+1}}{s_0!}p^{s_0}\text{sgn}(p),\ s_0\in\N.
\end{split}
\ee
An equation of the kind \eqref{g-ls-v2} still holds:
\be\label{g-ls-v3}\begin{split}
g(y)\sim & a^+(\bar y)P_+^{\gao}(y)+a^-(\bar y)P_-^{\gao}(y),\
a^\pm(\bar y)=\frac{\fs^+(\bar y)}{2s_0!}e(\mp (s_0+1)),\ s_0\in\N.
\end{split}
\ee
Combining \eqref{g-FT}--\eqref{g-Psi_N} and using that \eqref{g-Psi_nN} and \eqref{g-Psi_N} can be differentiated proves \eqref{Psic-ass-g}.

From the second equation in \eqref{g-FT}, \eqref{g-Psi_N}, and \eqref{hz-sp} we get also
\be\label{conorm}
\pa_y^mG(y,\cdot)\in 
\begin{cases}
C_*^{s_0}(\br) \text{ for any }m\in\N_0^n,\\ 
C_0^{s_0}(\br) \text{ for any }m\in\N_0^n\text{ if }  s_0\in\mathbb N.
\end{cases}
\ee
Together with the first equation in \eqref{g-FT} this proves \eqref{g-continuity}.

If $\fs^\pm\equiv0$, the result follows from the properties of $\tilde R(y,\la)$ and \eqref{hz-sp}, because $s_1\not\in\N$.


\subsection{Proof of Lemma \ref{lem:P_c-ass}}\label{subsec:prfPc}
From \eqref{psDO} and \eqref{g-def},
\be\label{Psi_c-expr}\begin{split}
(\CB g)(y)&=\frac1{(2\pi)^{n+1}}\int_{\br^n} \tilde B(y,\eta)\int_{\vs}\int_{\br}\fs(z,\la)e^{-i\la P(z)+i (z-y)\cdot\eta}d\la dz d\eta.
\end{split}
\ee
As is standard (see e.g., \cite{trev2}), set $u=\eta/\la$ and consider the phase function
\be
W(z,u,y):=P(z)-(z-y)\cdot u.
\ee
The only critical point $(z_0,u_0)$ and the corresponding Hessian $H$ are given by
\be\label{crit-pt}
z_0=y,\ u_0=P'(y),\ H=\bma P''(y) & -I \\ -I & 0\ema.
\ee
Clearly, $|\det H|=1$ and $\text{sgn}\, H=0$. By the stationary phase method we get using \eqref{B-lim} and \eqref{g-lim}
\be\label{J-st1}\begin{split}
J(y,\la):=&\frac{|\la|^n}{(2\pi)^n}\int_{\br^n}\int_{\vs} \tilde B(y,\la u)\fs(z,\la)e^{-i\la(P(z)-P(y) - (z-y)\cdot u)}dz du\\
=&\tilde B(y,\la P'(y))\fs(y,\la)+\tilde R(y,\la),\ \tilde R\in S^{\bt_0-s_0-2}(\vs\times\br).
\end{split}
\ee
The fact that $u$-integration is over an unbounded domain does not affect the result, because integrating by parts with respect to $z$ we obtain a function that decreases rapidly as $|u|\to\infty$ and $\la\to\infty$.

Similarly, considering the integral with respect to $z$ in \eqref{Psi_c-expr} and integrating by parts using property g1 of $g$ in Section~\ref{sec:assns} we get a function that rapidly decreases as $|\eta|\to\infty$ provided that $|\la|$ is bounded. Therefore, by property $\CB2$, 
\be\label{J-bnd}
\pa_y^m J(y,\la)\in L^\infty(\vs\times[0,c]),\ m\in\N_0^n,
\ee
for any $c>0$.

Substituting \eqref{J-st1} into \eqref{Psi_c-expr} and using \eqref{B-lim}, \eqref{g-lim} leads to
\be\label{Psi_c-f-FT}\begin{split}
(\CB g)(y)=&\frac1{2\pi}\int J(y,\la) e^{-i\la P(y)} d\la \\
=&\CF_{1d}^{-1}\biggl(\tilde B_0(y,P'(y))\fs^+(y)\la_+^{\bt_0-s_0-1}\\
&\hspace{1cm}+\tilde B_0(y,-P'(y))\fs^-(y)\la_-^{\bt_0-s_0-1}+\tilde R(y,\la)\biggr)(P(y)),\\
\tilde R\in &S^{c-1}(\vs\times\br),\ c:=\max(\bt_0-s_0-1,\bt_1-s_0,\bt_0-s_1)<\bt_0-s_0. 
\end{split}
\ee
The factor $|\la|^n$ in \eqref{J-st1} cancels because $|\la|^n du=d\eta$. Computing the asymptotics of the inverse Fourier transform as $P=P(y)\to0$ and using that $\tilde B_0(y,\pm P'(y))\in\coi(\vs)$ and $P(y)\equiv0$ on $\Gamma$ gives
\be\label{Psi_c-f}\begin{split}
(\CB g)(y)
=&c_1^+\Psi_{\bt_0-s_0}^+(P)+c_1^-\Psi_{\bt_0-s_0}^-(P)+R(y,P),\ P=P(y),\\ 
c_1^\pm=&\tilde B_0(\bar y,\pm P'(\bar y))\fs^\pm(\bar y),\ \bar y:=(\psi(y_\perp),y_\perp)\in\Gamma.
\end{split}
\ee
Recall that $\Psi_a^{\pm}$ are defined in \eqref{Psi_a}.
Here $R(y,p)$ is the remainder, which, by Theorem 5.12 in \cite{Abels12}, \eqref{J-bnd}, and the smoothness of $\tilde B_0$, satisfies
\be\label{Bg-rmd}
|\pa_y^m\pa_p^l R(y,p)|\le c_{m,l}\begin{cases} |p|^{c-l},& c<l,\\
1+|\log|p||,& c=l,\\ 1,& c>l,\end{cases}\ m\in\N_0^n,l\in\N_0,
\ee
for some $c_{m,l}>0$. The constant $c$ here is the same as in \eqref{Psi_c-f-FT}. The estimate \eqref{Psic-ass-Bg} follows from \eqref{Psi_c-f}, \eqref{Bg-rmd}. 

If $\kappa=0$, condition \eqref{c-ratios-req} implies 
\be\label{Bvsum-k0}\begin{split}
&\tilde B_0(\bar y,P'(\bar y))\fs^+(\bar y)\la_+^{\bt_0-s_0-1}
+\tilde B_0(\bar y,-P'(\bar y))\fs^-(\bar y)\la_-^{\bt_0-s_0-1}\\
&=\tilde B_0(\bar y,P'(\bar y))\fs^+(\bar y)(\la-i0)^{\bt_0-s_0-1},
\end{split}
\ee
and
\be\label{Bg-k0}\begin{split}
(\CB g)(y)
=&\tilde B_0(\bar y,P'(\bar y))\fs^+(\bar y)\begin{cases}\frac{e(-\bt_0+s_0+1)}{\Gamma(-\bt_0+s_0+1)}P(y)_-^{-(\bt_0-s_0)},& \bt_0-s_0\not\in\N\\
0,& \bt_0-s_0\in\N\end{cases}\\
&+R(y,P(y)),\ y\in\vs\setminus\Gamma.
\end{split}
\ee
This proves \eqref{Psic-ass-extra}.

\subsection{Proof of Lemma \ref{lem:gel}}\label{subsec:gel}

Using that $l>s_0$ and $\ik$ has $\lceil \bt_0^+ \rceil$ bounded derivatives and $\ik$ is exact to the degree $\lceil \bt_0 \rceil$, we get with any $0\le M\le l$:
\be\label{ge-Taylor}\begin{split}
g_\e^{(l)}(y)
&=\sum_j\pa_{y_1}^l\ik\left(\frac{y-\vy^j}\e\right)\left(g(\vy^j)-\sum_{|m|\le M-1}\frac{(\vy^j-y)^m}{m!}g^{(m)}(y)\right)\\
&=\e^{-l}\sum_j(\pa_{\check y_1}^l\ik) \left(\frac{y-\vy^j}\e\right) \sum_{|m|=M}R_m(\vy^j,y)(\vy^j-y)^{m},
\end{split}
\ee
where the remainder satisfies
\be\label{Taylor-rem}\begin{split}
R_m(\vy^j,y)&= \frac{|m|}{m!}\int_0^1 (1-t)^{|m|-1}g^{(m)}(y+t(\vy^j-y))dt,\\ 
|R_m(\vy^j,y)|&\le \frac1{m!}\max_{|m'|=|m|,\check y\in\text{supp} \ik}|g^{(m')}(y+\e \check y)|.
\end{split}
\ee

To prove the top case in \eqref{F-ineq-lot} select $\varkappa_1>0$ so that $|P(y)|\ge \varkappa_1\e$ and $\ik((y-w)/\e)\not=0$ implies $|P(w)|\ge \e$. By Lemma~\ref{lem:conorm}, this ensures that for each $m\in\N_0^{n}$ there exists $c(m)>0$ such that
\be\label{maxder}\begin{split}
\max_{\check y\in\text{supp} \ik}|g^{(m)}(y+\e \check y)|\le c(m)\begin{cases} |P(y)|^{s_0-|m|},&|m|>s_0,\\
1,&|m|\le s_0,\end{cases},\ |P(y)|\geq \varkappa_1\e.
\end{split}
\ee
Set $M=l$ in \eqref{ge-Taylor}. Then \eqref{maxder} together with the bottom line in \eqref{Taylor-rem} prove the result.

To prove the middle case in \eqref{F-ineq-lot}, set $M=\lfloor s_0 \rfloor$ in \eqref{ge-Taylor}.
If $s_0\in\mathbb N$, \eqref{g-continuity} and \eqref{Taylor-rem} imply that $R_m=O(1)$, thereby proving the assertion. If $s_0\not\in\N$, the remainder can be modified as follows
\be\label{ge-Taylor-alt}\begin{split}
\tilde R_m(\vy^j,y)&= \frac{|m|}{m!}\int_0^1 (1-t)^{|m|-1}[g^{(m)}(y+t(\vy^j-y))-g^{(m)}(y)]dt=O(\e^{\{s_0\}}).
\end{split}
\ee
Here we have used \eqref{holder} with $r=s_0$. Since $l>M$, we can replace $R_m$ with $\tilde R_m$ in \eqref{ge-Taylor} without changing the equality, and the desired inequality follows.

The bottom case in \eqref{F-ineq-lot} follows by setting $M=l$ in \eqref{ge-Taylor} and noticing that  \eqref{g-continuity} and \eqref{Taylor-rem} imply $R_m=O(1)$.

If $\fs^\pm\equiv0$, the same argument as above applies with $s_0$ replaced by $s_1$. The only change is that there is no need to consider the case $s_1\in\N$.

\subsection{Proof of Lemma \ref{lem:delta}}\label{subsec:delta}
Since $\varkappa_1>0$ is the same as in the proof of Lemma \ref{lem:gel}, $|P(y)|\ge \varkappa_1\e$ and $\ik((y-w)/\e)\not=0$ imply $|P(w)|\ge \e$. Similarly to \eqref{ge-Taylor}, using the properties of $\ik$ we obtain
\be\label{ge-Taylor-v2}\begin{split}
g_\e^{(l)}(y)
&=\sum_j\pa_{y_1}^l\ik\left(\frac{y-\vy^j}\e\right)\biggl(\sum_{|m|\le M-1}\frac{(\vy^j-y)^m}{m!}g^{(m)}(y)\\
&\hspace{4cm}+\sum_{|m|=M}R_m(\vy^j,y)(\vy^j-y)^{m}\biggr)\\
&=g^{(l)}(y)+\sum_j\pa_{y_1}^l\ik\left(\frac{y-\vy^j}\e\right)\sum_{|m|=M}R_m(\vy^j,y)(\vy^j-y)^{m},\ l<M\le\lceil\bt_0\rceil
\end{split}
\ee
The term $g^{(l)}(y)$ on the right in \eqref{ge-Taylor-v2} is the only term from the Taylor polynomial that remains after the summation with respect to $j$. In particular, all the terms corresponding to $l< |m| \le M-1$ are converted to zero, because $\ik$ is exact to the degree $\lceil\bt_0\rceil$, and 
\be\label{conv-to-0}
\sum_j\pa_{y_1}^l\ik\left(\frac{y-\vy^j}\e\right) (\vy^j-y)^m=\pa_{w_1}^l (w-y)^m|_{w=y}=0,\ l< |m| \le M-1.
\ee
Using \eqref{ge-Taylor-v2} with $M=l+1$ and appealing to \eqref{Taylor-rem}, \eqref{maxder} proves \eqref{F-ineq-1}. Indeed, recall that $l\ge\lfloor s_0^-\rfloor$, so $M=l+1\ge s_0$.
If $s_0\not\in\N$, then $M>s_0$, and the top case in \eqref{maxder} applies when estimating 
$R_m$, $|m|=M$. If $s_0\in\N$, then $M=s_0$, and the bottom case in \eqref{maxder} applies when estimating $R_m$, $|m|=M$.

To prove \eqref{F-ineq-2}, we use \eqref{ge-Taylor-v2} with $M=\lfloor s_0\rfloor$. If $s_0\in\N$, then $l<\lfloor s_0\rfloor=s_0$ (by assumption, $l\le \lfloor s_0^-\rfloor$), and \eqref{ge-Taylor-v2}, \eqref{g-continuity} prove \eqref{F-ineq-2}. 

If $s_0\not\in\N$, we replace $R_m$ with $\tilde R_m$ in \eqref{ge-Taylor-v2} as this was done in the proof of Lemma~\ref{lem:gel}. As before, this does not invalidate the equality and extends its applicability to the case $l=M$. Note, however, that if $l=M$, then the term $g^{(l)}(y)$ on the right in \eqref{ge-Taylor-v2} comes not from the Taylor polynomial, but from the modification of the remainder. The desired assertion follows from \eqref{ge-Taylor-alt} and the modified \eqref{ge-Taylor-v2}.

If $\fs^\pm\equiv0$, the same argument as above applies with $s_0$ replaced by $s_1$. The only change is that there is no need to consider the case $s_1\in\N$.

\section{Proof of lemma \ref{key-sum-lemma}}\label{sec:lem3prf}

Throughout the proof, $c$ denotes various positive constants that can vary from one place to the next. To simplify notations, in this proof we drop the subscripts from $\bt_0$ and $s_0$:  $\bt=\bt_0$, $s=s_0$. By the choice of $y$ coordinates (see \eqref{y-coords}) and by \eqref{Theta-def}, $y_1=\bp\cdot y$ (recall that $|\bp|=1$).

Starting from \eqref{I-def-est} we estimate the difference between the terms with the subscript `+' on the left and on the right as follows 
\be\label{g-lt-est}\begin{split}
\bigl|a^+(\vy^j)&P_+^{s}(\vy^j)-a^+(y_0)(\vy^j_1-z_1)_+^s\bigr|\\
\le &\left|P_+^{s}(\vy^j)-(\vy^j_1-z_1)_+^s\right| |a^+(\vy^j)|
+|\vy^j_1-z_1|^s\left|a^+(\vy^j)-a^+(y_0)\right|,\\ 
\vy^j=&U^T(\e D j-\tilde y_0).
\end{split}
\ee
The following inequalities can be shown to hold. For all $q,r\in\br$ one has
\be\label{two-ineqs}\begin{split}
\left|(q+r)_\pm^s-q_\pm^s\right|\le & 2^{s-1}(|r|^s+s|q|^{s-1}|r|),\ s> 1,\\
\left|(q+r)_\pm^s-q_\pm^s\right|\le & |r|^s,\ 0<s\le 1.
\end{split}
\ee
Consider the top inequality. The case $q,q+r\le0$ is trivial. The cases $q+r\le 0\le q$ and $q\le 0\le q+r$ can be verified directly. By a change of variables and convexity, it is easily seen that the case $r<0< q$ follows from the case $q,r>0$. To prove the latter, divide by $q^s$ and set $x=r/q$. Both sides equal zero when $x=0$. Differentiating with respect to $x$, we see that the inequality is proven because $(1+x)^{s-1}\le 2^{s-1}(x^{s-1}+1)$ (consider $0<x\le 1$ and $x\ge 1$). The second inequality in \eqref{two-ineqs} is obvious.

The assumption $z\in\Gamma$ implies $z_1=\psi(z_\perp)$, so
\be\label{P-arg-diff}
P(\vy^j)=\vy^j_1-\psi(\vy^j_\perp)=\vy^j_1-z_1+(\psi(z_\perp)-\psi(\vy^j_\perp)).
\ee
Setting $q=\vy^j_1-z_1$ and $r=\psi(z_\perp)-\psi(\vy^j_\perp)$ in \eqref{two-ineqs} and using \eqref{P-arg-diff} and that $a^+(y)$ is bounded, we estimate the first term on the right in \eqref{g-lt-est} as follows
\be\label{ftr}\begin{split}
&\left|P_+^{s}(\vy^j)-(\vy^j_1-z_1)_+^s\right| |a^+(\vy^j)|\\
&\le c \left(|\psi(z_\perp)-\psi(\vy^j_\perp)|^s+
\begin{cases}
|\vy^j_1-z_1|^{s-1}|\psi(z_\perp)-\psi(\vy^j_\perp)|,& s> 1\\
0,&0<s\le 1
\end{cases}
\right).
\end{split}
\ee
Recall that in this lemma we assume that the amplitude of $\CB$ satisfies $\tilde B(y,\eta)\equiv\tilde B_0(y,\eta)$. By \eqref{B-lim}, the fact that the amplitude of $\CB$ is homogeneous in the frequency variable (and, therefore, the Schwartz kernel $K(y,w)$ of $\CB$ is homogeneous in $w$), and Assumption IK1,
\be\label{Bik-bnd}
|\CB \ik(u)|\leq c\left(1+|u|\right)^{-(\bt+n)},\ u\in\br^n.
\ee
Therefore, by \eqref{B-expl} and \eqref{I-def-est}, we have to estimate the following two sums
\be\label{J123}\begin{split}
J_1:=&\sum_{|j|\le O(1/\e)}\frac{|(\psi(z_\perp)-\psi(\vy^j_\perp))/\e|^s}{(1+|(y-\vy^j)/\e|)^{\bt+n}},\\
J_2:=&\sum_{|j|\le O(1/\e)}\frac{|(\vy_1^j-z_1)/\e|^{s-1}|(\psi(z_\perp)-\psi(\vy^j_\perp))/\e|}{(1+|(y-\vy^j)/\e|)^{\bt+n}}.
\end{split}
\ee
The second sum is required if $s>1$. Assumptions of the lemma imply 
\be\label{zpsi-est}
|\psi(z_\perp)-\psi(\vy^j_\perp)|\le |\psi'(y_\perp^*)||z_\perp-\vy^j_\perp|
\le c(\e^{1/2}+|z_\perp-\vy^j_\perp|)|z_\perp-\vy^j_\perp|
\ee
for some $c>0$. Here $y^*$ is some point on the line segment with the endpoints $z_\perp$, $\vy^j_\perp$, and we have used that $|\psi'(y^*)|\le c(|z_\perp|+|z_\perp-\vy^j_\perp|)$, which follows from $\psi'(y_0)=0$.

Let $m=m(z,\e)\in\mathbb Z^n$ be such that $|(z+U^T\tilde y_0)/\e-U^TDm|<c$. The dependence of $m$ on $z$ and $\e$ is omitted from notations. This implies
\be\label{zydist}\begin{split}
\max(|z_1-\vy^j_1|,|z_\perp-\vy^j_\perp|)&\le |z-\vy^j|\le c \e \left|\frac{z+U^T\tilde y_0}{\e}-U^TDm-U^T D (j-m)\right|\\
&\le c\e(c+|j-m|).
\end{split}
\ee
Also, using that $|y-z|=O(\e)$ gives
\be\label{yyjdist}
\left|\frac{y-\vy^j}\e\right|=\left|\frac{(y-z)+(z-\vy^j)}\e\right| \ge c|j-m|\text{ if } |j-m|\gg 1.
\ee
Substitute \eqref{zpsi-est} into the expression for $J_1$ in \eqref{J123}, shift the index $j\to j-m$, and use \eqref{zydist}, \eqref{yyjdist}:
\be\label{J123-st1}\begin{split}
J_1\le &c\sum_{|j|\le O(1/\e)}\frac{(\e^{1/2}+\e (c+|j|))^s(c+|j|)^s}{(1+c|j|)^{\bt+n}}+O(\e^{s/2}).
\end{split}
\ee
Here we have used that we can ignore any finite number of terms (their contribution is $O(\e^{s/2})$), and \eqref{yyjdist} applies to the remaining terms. This gives
\be\label{J123-st2}\begin{split}
J_1\le & c\sum_{0<|j|\le O(1/\e)}\frac{(\e^{1/2}+\e |j|)^s}{|j|^{\bt+n-s}}+O(\e^{s/2})\\
\le &c\int_1^{O(1/\e)}\frac{(\e^{1/2}+\e r)^s}{r^{\bt+1-s}}dr+O(\e^{s/2})=O(\e^{\min(\bt-s,s/2)}).
\end{split}
\ee

To estimate $J_2$, we use the same approach as in \eqref{zpsi-est} -- \eqref{J123-st2}:
\be\label{J2-st1}\begin{split}
J_2\le &c\sum_{|j|\le O(1/\e)}\frac{(\e^{1/2}+\e |j|)(c+|j|)^s}{(1+c|j|)^{\bt+n}}+O(\e^{1/2})\\
\le & c\sum_{0<|j|\le O(1/\e)}\frac{\e^{1/2}+\e |j|}{|j|^{\bt+n-s}}+O(\e^{1/2})\\
\le &c\int_1^{O(1/\e)}\frac{\e^{1/2}+\e r}{r^{\bt+1-s}}dr+O(\e^{1/2})=O(\e^{\min(\bt-s,1/2)})=O(\e^{1/2}).
\end{split}
\ee
Here we have used that $\bt-s\ge N/2\ge 1/2$. 

The second term on the right in \eqref{g-lt-est} is estimated as follows:
\be\label{sec-term}\begin{split}
&|\vy^j_1-z_1|^s\left|a^+(\vy^j)-a^+(y_0)\right|
\le |\vy^j_1-z_1|^s\left|(a^+(\vy^j)-a^+(z))+(a^+(z)-a^+(y_0))\right| \\
&\le c[\e(c+|j-m|)]^s(\e(c+|j-m|)+\e^{1/2}).
\end{split}
\ee
Shifting the $j$ index as before and estimating a finite number of terms by $O(\e^{1/2})$ gives an upper bound
\be\label{J3}
\sum_{0<|j|\le O(1/\e)}\frac{\e^{1/2}+\e|j|}{|j|^{\bt+n-s}}+O(\e^{1/2})
=O(\e^{1/2}).
\ee
The terms with the subscript $'-'$ in \eqref{I-def-est} are estimated analogously. Our argument proves \eqref{I-def-est} with $\CB$ instead of $\CB_0$ on the right. 

The left-hand side of \eqref{I-def-est} is bounded, because 
\be\label{Pest}\begin{split}
|P(\vy^j)|\le |\vy^j_1-z_1|+|\psi(z_\perp)-\psi(\vy^j_\perp)|
\le & c\e(c+|j-m|)(1+(\e^{1/2}+\e(c+|j-m|)))\\
\le & c\e(1+|j-m|)
\end{split}
\ee
by \eqref{P-arg-diff}, \eqref{zpsi-est}, \eqref{zydist}, and $|j|\le O(1/\e)$, and
\be\label{temp1}\begin{split}
\frac{|P(\vy^j)/\e|^s}{(1+|(y-\vy^j)/\e|)^{\bt+n}}
\le & c\sum_{0<|j|\le O(1/\e)}\frac1{|j|^{\bt+n-s}}+O(1)<\infty.
\end{split}
\ee

It is easy to see that
\be\label{bb0}\begin{split}
\e^{\bt}\left|\left(\CB \ik\left(\frac{\cdot-\vy^j}\e\right)\right)(y) -
\CB_0 \ik\left(\frac{y-\vy^j}{\e}\right)\right|\le c\frac{\e^{1/2}}{(1+|(y-\vy^j)/\e|)^{\bt+n}}.
\end{split}
\ee
This follows from $\ik\in C_0^{\lceil\bt_0^+\rceil}$, $|y-y_0|=O(\e^{1/2})$, and
\be\label{DelB-est}
|\pa_\eta^m(\tilde B_0(y,\eta)-\tilde B_0(y_0,\eta))|\le c_m|y-y_0||\eta|^{\bt-|m|},\ |\eta|\ge 1,\ m\in\N_0^n.
\ee
Together with \eqref{temp1} this implies that replacing $y$ with $y_0$ in the amplitude of the $\Psi$DO $\CB$ (i.e., replacing $\tilde B_0(y,\eta)$ with $\tilde B_0(y_0,\eta)$) introduces an error of the magnitude $O(\e^{1/2})$, and the lemma is proven.

\section{Proof of lemma \ref{lem:psi-incr}}\label{sec:delphi}

Pick some sufficiently large $J\gg 1$. Then, with $D_1:=U^TD$, 
\be\label{an-psi-pr-2}\begin{split}
\Mpsi(\check y+\Delta \check y,p)-\Mpsi(\check y,p)&= \sum_{j\in\mathbb Z^n} \bigl[\CB_0 \ik(\check y+\Delta \check y-D_1 j)-\CB_0 \ik(\check y-D_1 j)\bigr]\CA(\bp\cdot D_1 j-p)\\
&=\sum_{|j|\le J}(\cdot)+\sum_{|j|>J}(\cdot)=:J_1+J_2.
\end{split}
\ee
Because $\CB_0\in S^{\bt_0}(\br^n)$, Theorem 6.19 in \cite{Abels12} implies that $\CB_0:C_*^{\lceil\bt_0^+\rceil}\to C_*^a$, $a=\lceil\bt_0^+\rceil-\bt_0=1-\{\bt_0\}>0$, is continuous. Nonsmoothness of the symbol at the origin, which is not allowed by the assumptions of the theorem, is irrelevant. By assumption, $\ik\in C_0^{\lceil\bt_0^+\rceil}(\br^n)$, so $J_1=O(|\Delta \check y|^a)$. In the second term $J_2$, the arguments of $\CB_0\ik$ are bounded away from zero, and the factor in brackets is smooth. Moreover, using again that the Schwartz kernel $K(y,w)$ of $\CB_0$ is homogeneous in $w$, we have,
\be\label{estim-der}
|\pa_{u_l}\CB_0 \ik(u)|=O(|u|^{-(n+\bt_0+1)}),\ |u|\to\infty,\ 1\leq l\leq n.
\ee
Using the argument analogous to the one in \eqref{temp1}, we easily see that $J_2=O(|\Delta \check y|)$. This proves the first line in \eqref{an-psi-pr-1}.

The second line in \eqref{an-psi-pr-1} is proven analogously:
\be\label{an-psi-pr-3}\begin{split}
&\Mpsi(\check y,p+\Delta p)-\Mpsi(\check y,p)\\
&\quad= \sum_{j\in\mathbb Z^n} \CB_0 \ik(\check y-D_1 j)\bigl[\CA(\bp\cdot D_1 j-(p+\Delta p))-\CA(\bp\cdot D_1 j-p)\bigr]\\
&\quad=\sum_{|\bp\cdot D_1 j|\le J}(\cdot)+\sum_{|\bp\cdot D_1 j|>J}(\cdot)=:J_1+J_2.
\end{split}
\ee
Clearly, $\CA(q+\Delta p)-\CA(q)=O(|\Delta p|^{\min(s_0,1)})$ uniformly in $q$ confined to any bounded set. Using in addition that $\CB_0 \ik(u)$ is bounded and $\CB_0 \ik(u)=O(|u|^{-(n+\bt_0)})$ as $|u|\to\infty$, we get that $J_1=O(|\Delta p|^{\min(s_0,1)})$.

In $J_2$, the argument of $\CA$ is bounded away from zero. In view of $\CA'(q)=O(|q|^{s_0-1})$, $|q|\to\infty$, we finish the proof by noticing that
\be\label{j2-est-del}
|J_2| \leq O(|\Delta p|) \sum_{|j| > 0}\frac{|j|^{s_0-1}}{|j|^{n+\bt_0+1}}=O(|\Delta p|) .
\ee

The fact that both estimates are uniform with respect to $\check y$ and $p$ confined to bounded sets is obvious.

\section{Proof of lemma \ref{lem:pPdiff}}\label{sec:pPdiff}

As usual, $c$ denotes various positive constants that may have different values in different places. Recall that $\bt_0-s_0>0$. Set $k:=\lceil\bt_0\rceil$, $\nu:=k-\bt_0$. Thus, $0\le\nu< 1$, and $\nu=0$ if $\bt_0\in\mathbb N$. Similarly to \eqref{B-split-lot}, 
\be\label{B-split}
\CB =\CW_1 \pa_{y_1}^k+\CW_2,
\ee
for some $\CW_1\in S^{-\nu}(\vs)$, $\CW_2\in S^{-\infty}(\vs)$. 

\subsection{Proof in the case $\bt_0\not\in\N$}
Let $K(y,w)$ be the Schwartz kernel of $\CW_1$. Suppose, for example, that $P:=P(y)>0$. The case $P<0$ is completely analogous. Initially, as $\varkappa_2$ in Lemma~\ref{lem:pPdiff}, we can pick any constant that satisfies $\varkappa_2\ge 2\varkappa_1$, where $\varkappa_1$ is the same as in \eqref{F-ineq-1}. This implies that $P/2\ge \varkappa_1\e$. Later (see the beginning of the proof of Lemma~\ref{lem:aux}), we update the choice of $\varkappa_2$. Denote (cf. \eqref{bt-aux-est})
\be\label{Je-def}
J_\e(y):=(\CB g_\e)(y)-(\CB g)(y)=(\CB\Delta g_\e)(y).
\ee
Then
\be\label{J-e1}\begin{split}
J_\e(y)=&J_\e^{(1)}(y)+J_\e^{(2)}(y)+O(\e^{s_0}),\\
J_\e^{(1)}(y):=&\int_{|P(w)|\geq P/2}  K(y,y-w)\Delta g_\e^{(k)}(w)dw,\\
J_\e^{(2)}(y):=&\int_{|P(w)|\leq P/2} K(y,y-w)\Delta g_\e^{(k)}(w)dw.
\end{split}
\ee
The big-$O$ term in \eqref{J-e1} appears because of the $\Psi$DO $\CW_2$ in \eqref{B-split}, and the magnitude of the term follows from \eqref{F-ineq-2} with $l=0$. From \eqref{F-def} and \eqref{F-ineq-1} with $l=k$, \eqref{J-e1}, and \eqref{ker-est-lot} with $l=0$, it follows that
\be\label{J-etwo}\begin{split}
|J_\e^{(1)}(y)|\leq &c\e\int_{|P(w)|\geq P/2} \frac{|w_1-\psi(w_\perp)|^{s_0-1-k}}{|y-w|^{n-\nu}} dw\\
=&c\e\int_{|p|\geq P/2}\int \frac{|p|^{s_0-1-k}}{|([P-p]+\psi(y_\perp)-\psi(w_\perp),y_\perp-w_\perp)|^{n-\nu}}dw_\perp dp.
\end{split}
\ee
Hence, we obtain similarly to \eqref{J1-ethree-lot}
\be\label{J1-ethree}\begin{split}
|J_\e^{(1)}(y)|\leq &c\e\int\int_{|p|\geq P/2} \frac{|p|^{s_0-1-k}}{(|P-p|+|w_\perp|)^{n-\nu}} dpdw_\perp\\
\leq & c\e \int_{|p|\geq P/2} \frac{|p|^{s_0-1-k}}{|P-p|^{1-\nu}} dp=c\e P^{s_0-1-\bt_0}.
\end{split}
\ee

To estimate $J_\e^{(2)}(y)$, integrate by parts with respect to $w_1$ in \eqref{J-e1}:
\be\label{J2-ezero}\begin{split}
|J_\e^{(2)}(y)|\leq & c\left(J_k+\sum_{l=l_0}^{k-1} (J_l^-+J_l^+)\right),\
l_0:=\lfloor s_0^-\rfloor,\\
J_k:=&\int_{|P(w)|\leq P/2} \left|\pa_{w_1}^{k-l_0}K(y,y-w)\Delta g_\e^{(l_0)}(w)\right|dw,\\
J_l^\pm:=&\int_{\br^{n-1}} \left|\pa_{w_1}^{k-1-l}K(y,y-w)\Delta g_\e^{(l)}(w)\right|_{w=(\psi(w_\perp)\pm P/2,w_\perp)}dw_\perp.
\end{split}
\ee
By construction, $P/2 \ge \varkappa_1\e$. Using \eqref{F-ineq-1}, \eqref{F-ineq-2} with $l=l_0$ (both inequalities apply when $l=l_0=\lfloor s_0^-\rfloor$), and arguing similarly to \eqref{J-etwo}, \eqref{J1-ethree}, gives
\be\label{J2k-1}\begin{split}
J_k\leq & c \int_{|p|\leq \varkappa_1\e} \frac{\e^{s_0-l_0}}{|P-p|^{\bt_0+1-l_0}} dp+
c\e \int_{\varkappa_1\e \le |p|\leq P/2} \frac{|p|^{s_0-l_0-1}}{|P-p|^{\bt_0+1-l_0}} dp\\
\le & c\e^{s_0-l_0+1} P^{-(\bt_0+1-l_0)}+ c\e P^{s_0-1-\bt_0}\int_{\varkappa_1(\e/P) \le |\check p|\leq 1/2} \frac{|\check p|^{s_0-l_0-1}}{|1-\check p|^{\bt_0+1-l_0}} d\check p\\
\le & c\e^{s_0-l_0+1} P^{-(\bt_0+1-l_0)}+ c\e P^{s_0-1-\bt_0}\left(1+(\e/P)^{s_0-l_0}\right),
\end{split}
\ee
where we have used that $l_0<s_0$. Using again that $\e/P \le 1/(2\varkappa_1)$ gives $J_k\le c\e P^{s_0-1-\bt_0}$.

Next we estimate the boundary terms in \eqref{J2-ezero}. By \eqref{F-ineq-1} (using that $\lfloor s_0^-\rfloor=l_0\le l\le k-1$) and \eqref{ker-est-lot}, 
\be\label{Jb-est-v1}\begin{split}
J_l^\pm\leq c\e\int_{\br^{n-1}} \frac{|w_1-\psi(w_\perp)|^{s_0-1-l}}{|y-w|^{n+\bt_0-1-l}} dw_\perp,\ w_1=\psi(w_\perp)\pm P/2.
\end{split}
\ee
Appealing to \eqref{denom-est} gives
\be\label{Jb-est-v2}\begin{split}
J_l^\pm \leq c\e P^{s_0-1-l}\int_{\br^{n-1}} \frac{dw_\perp}{(P\pm(P/2)+|w_\perp|)^{n+\bt_0-l-1}}=c\e P^{s_0-1-\bt_0},
\end{split}
\ee
which finishes the proof. As easily checked, the integral in \eqref{Jb-est-v2} converges because $l\le k-1 <\bt_0$.

\subsection{Proof in the case $\bt_0\in\N$} Suppose now $\bt_0\in\N$, i.e. $k=\bt_0$ and $\nu=0$. All the terms that do not involve integration over a neighborhood of the set $\{w\in\vs:\,P(w)=P\}$ are estimated the same way as before. For example, estimation of $J_\e^{(2)}(y)$ is completely analogous to \eqref{J2-ezero}--\eqref{Jb-est-v2}, and we obtain the same bound $|J_\e^{(2)}(y)|\le c\e P^{s_0-1-\bt_0}$. Estimating of $J_\e^{(1)}$ is much more involved now, because the singularity at $P(w)=P$ is no longer integrable. We have with some $c_1>0$, which is to be selected later:
\be\label{J-e1-parts}\begin{split}
J_\e^{(1)}(y)=&J_\e^{(1a)}(y)+J_\e^{(1b)}(y)+J_\e^{(1c)}(y),\\
J_\e^{(1a)}(y):=&\int_{P/2\leq P(w) \leq P-c_1\e}  K(y,y-w)\Delta g_\e^{(\bt_0)}(w)dw,\\
J_\e^{(1b)}(y):=&\int_{P-c_1\e\leq P(w) \leq P+c_1\e}  K(y,y-w)\Delta g_\e^{(\bt_0)}(w)dw,\\
J_\e^{(1c)}(y):=&\int_{P+c_1\e\leq P(w)}  K(y,y-w)\Delta g_\e^{(\bt_0)}(w)dw.
\end{split}
\ee
We do not estimate the integral $\int_{P(w)\le P/2} (\cdot)dw$, because the domain of integration is bounded away from the set $\{w\in\vs:\,P(w)=P\}$, and this integral admits the same bound as in the previous subsection (cf. \eqref{J1-ethree}). Similarly to \eqref{J1-ethree}, by \eqref{F-ineq-1} with $l=\bt_0$,
\be\label{Je1a}\begin{split}
|J_\e^{(1a)}(y)|\leq  & c\e \int_{P/2\le p\le P-c_1\e } \frac{p^{s_0-1-\bt_0}}{P-p} dp=c\e P^{s_0-1-\bt_0}\ln(P/\e),\\
|J_\e^{(1c)}(y)|\leq  & c\e \int_{P+c_1\e \le p} \frac{p^{s_0-1-\bt_0}}{p-P} dp=c\e P^{s_0-1-\bt_0}\ln(P/\e).
\end{split}
\ee
The term $J_\e^{(1b)}$ is split further as follows:
\be\label{Je1b-parts}\begin{split}
J_\e^{(1b)}(y)=&J_\e^{(1b1)}(y)+J_\e^{(1b2)}(y)+J_\e^{(1b3)}(y),\\
J_\e^{(1b1)}(y):=&\int_{\substack{|P-P(w)| \leq c_1\e\\ |y_\perp-w_\perp|\ge c_1P}}  K(y,y-w)\Delta g_\e^{(\bt_0)}(w)dw,\\
J_\e^{(1b2)}(y):=&\int_{\substack{|P-P(w)| \leq c_1\e\\ |y_\perp-w_\perp|\le c_1P}}  K(y,y-w)(\Delta g_\e^{(\bt_0)}(w)-\Delta g_\e^{(\bt_0)}(y))dw,\\
J_\e^{(1b3)}(y):=&\Delta g_\e^{(\bt_0)}(y)I,\ I:=\int_{\substack{|P-P(w)| \leq c_1\e\\ |y_\perp-w_\perp|\le c_1 P}} K(y,y-w)dw.
\end{split}
\ee
Similarly to \eqref{J1-ethree}, by \eqref{F-ineq-1} with $l=\bt_0$,
\be\label{J1b1est}\begin{split}
|J_\e^{(1b1)}(y)|\leq &c\e\int_{|w_\perp|\geq c_1P}\int_{|P-p| \leq c_1\e} \frac{p^{s_0-1-\bt_0}}{(|P-p|+|w_\perp|)^n} dpdw_\perp \leq c\e^2 P^{s_0-2-\bt_0}.
\end{split}
\ee
The second part is estimated by rearranging the $\Delta g$ terms:
\be\label{J1bdest}\begin{split}
J_\e^{(1b2)}(y):=\int_{\substack{|P-P(w)| \leq c_1\e\\ |y_\perp-w_\perp|\le c_1P}}  K(y,y-w)\bigl[&(g_\e^{(\bt_0)}(w)-g_\e^{(\bt_0)}(y))\\
&-(g^{(\bt_0)}(w)-g^{(\bt_0)}(y))\bigr]dw.
\end{split}
\ee

\begin{lemma}\label{lem:aux} There exist $c,c_1,\varkappa_2>0$ so that
\be\label{delge-ineqs}\begin{split}
&|g^{(\bt_0)}(w)-g^{(\bt_0)}(y)|\le  c|w-y|P(y)^{s_0-1-\bt_0},\\
&|g_\e^{(\bt_0)}(w)-g_\e^{(\bt_0)}(y)|\le c|w-y|P(y)^{s_0-1-\bt_0},\\
&\text{if}\quad |P(y)-P(w)| \leq  c_1\e,\ |y_\perp-w_\perp|\le c_1P(y),\ P(y)>\varkappa_2\e.
\end{split}
\ee
\end{lemma}
\begin{proof} We begin by updating the choice of $\varkappa_2$. Select $\varkappa_2\ge 2\varkappa_1$ so that $P\ge \varkappa_2\e$ implies
\be\label{C3-select}
P(v)\ge c P \text{ for any }v\in\vs, |v-y|\le \e d_\ik,\ d_\ik:=\text{diam}(\text{supp}\,\ik),
\ee
for some $c>0$. 

Next we select $c_1$. First, pick any $c_1$ so that $0<c_1\le \varkappa_1$. This ensures that $P(w)\ge P-|P-P(w)|\ge \varkappa_1\e$, and \eqref{F-ineq-1} can be used to estimate the derivatives of $\Delta g_\e(w)$.
Let $c_\psi:=\max_{v\in\vs}|\psi'(v)|$. Our assumptions imply 
\be
|y_1-w_1|\le |P-P(w)|+|\psi(y_\perp)-\psi(w_\perp)|\le c_1(\e+c_\psi P). 
\ee
Let $v$ be any point on the line segment with the endpoints $w$ and $y$, i.e. $v=y+\la(w-y)$, $0\le \la\le 1$. Then 
\be\label{P-bnd}
P(v)\ge P-(|y_1-w_1|+|\psi(v_\perp)-\psi(y_\perp)|)\ge P-c_1(\e+c_\psi P)-c_\psi c_1 P.
\ee
Reducing $c_1>0$ even further, we can ensure that $P(v)\ge cP$ for some $c>0$. 
This is the value of $c_1$ that is assumed starting from \eqref{J-e1-parts}. In the rest of the proof we assume that $w,y\in\vs$ satisfy the inequalities on the last line in \eqref{delge-ineqs} with the constants $c_1$ and $\varkappa_2$ that we have just selected.
 
From \eqref{Psic-ass-g} with $|m|=\bt_0+1$, 
\be\label{delg-bnd}
|g^{(\bt_0)}(w)-g^{(\bt_0)}(y)|\le |w-y|\max_{0\le \la\le 1} |\nabla g^{(\bt_0)}(y+\la(w-y))|
\le c|w-y|P^{s_0-1-\bt_0}
\ee
for some $c>0$.

To prove the second line in \eqref{delge-ineqs}, find $c_{2,3}>0$ such that
\be\label{c2-select}
v\in\vs,|v-y|\le \e(c_2+d_\ik) \text{ implies } P(v)\ge c_3 P.
\ee
By \eqref{C3-select}, $c_{2,3}$ with the required properties do exist.

Now, assume first that $|w-y|\ge c_2\e$, where $c_2$ is the same as in \eqref{c2-select}. Clearly,
\be\label{delge-triangle}\begin{split}
|g_\e^{(\bt_0)}(w)-g_\e^{(\bt_0)}(y)|\le |\Delta g_\e^{(\bt_0)}(w)|&+|g^{(\bt_0)}(w)-g^{(\bt_0)}(y)|+|\Delta g_\e^{(\bt_0)}(y)|.
\end{split}
\ee
By construction, \eqref{F-ineq-1} applies to $\Delta g_\e^{(\bt_0)}(w)$. Applying \eqref{F-ineq-1} to the first and third terms on the right in \eqref{delge-triangle}, and \eqref{delg-bnd} -- to the second term on the right, gives
\be\label{delge-triangle-cont1}\begin{split}
|g_\e^{(\bt_0)}(w)-g_\e^{(\bt_0)}(y)|&\le c\e P(w)^{s_0-1-\bt_0}+c|w-y|P^{s_0-1-\bt_0}+c\e P^{s_0-1-\bt_0}\\
&\le c|w-y|P^{s_0-1-\bt_0},
\end{split}
\ee
because $\e\le (1/c_2)|w-y|$ and 
\be\label{PwP-ineq}
P(w)\ge P-|P-P(w)|\ge P(1-c_1(\e/P))\ge P(1-c_1/(2\varkappa_1))\ge P/2.
\ee

If $|w-y|\le c_2\e$, we argue similarly to \eqref{ge-Taylor-v2}:
\be\label{delge-Taylor}\begin{split}
g_\e^{(\bt_0)}(w)-g_\e^{(\bt_0)}(y)
&=\e^{-\bt_0}\sum_j\left((\pa_{\check y_1}^{\bt_0}\ik)\left(\frac{w-\vy^j}\e\right)-(\pa_{\check y_1}^{\bt_0}\ik)\left(\frac{y-\vy^j}\e\right)\right)\\
&\hspace{1cm}\times\sum_{|m|=\bt_0+1}R_m(\vy^j,y)(\vy^j-y)^{m}\\
|R_m(\vy^j,y)|&\le \frac1{(\bt_0+1)!}\max_{|m'|=\bt_0+1,|y-v|\le (c_2+d_\ik)\e}|\pa_v^{m'} g(v)|.
\end{split}
\ee
By \eqref{c2-select}, \eqref{Psic-ass-g} implies $|R_m(\vy^j,y)|\le cP^{s_0-1-\bt_0}$. The assertion follows because $\ik\in C_0^{\lceil\bt_0^+\rceil}(\br^n)$.
\end{proof}

Applying \eqref{ker-est-lot} with $\nu=l=0$ and \eqref{delge-ineqs} in \eqref{J1bdest} yields (cf. \eqref{J1-lot}--\eqref{J1-ethree-lot})
\be\label{J1bdest-st2}\begin{split}
|J_\e^{(1b2)}(y)|\le &cP^{s_0-1-\bt_0}\int_{\substack{|P-P(w)| \leq c_1\e\\ |y_\perp-w_\perp|\le c_1 P}}  \frac{|w-y|}{|y-w|^n} dw\\
\leq &cP^{s_0-1-\bt_0}\int_{|w_\perp|\le c_1P}\int_{|P-p| \leq c_1\e} \frac{1}{(|P-p|+|w_\perp|)^{n-1}} dpdw_\perp\\
\leq & cP^{s_0-1-\bt_0}\int_{|w_\perp|\le c_1 P} \int_{|p| \leq c_1\e} \frac{dpdw_\perp}{(|p|+|w_\perp|)^{n-1}} =c\e P^{s_0-1-\bt_0}\ln\biggl(\frac{P}\e\biggr).
\end{split}
\ee
The final major step is to estimate the integral in the definition of $J_\e^{(1b3)}$. 
\be\begin{split}\label{I-sto}
I=& \int_{|y_\perp-w_\perp|\le c_1 P}\int_{|(y_1-w_1)-(\psi(y_\perp)-\psi(w_\perp)|\le c_1\e}K(y,(y_1-w_1,y_\perp-w_\perp)) dw_1dw_\perp\\
=&\int_{|v_\perp|\le c_1 P}\int_{-c_1 \e}^{c_1 \e}K(y,(v_1+h(v_\perp),v_\perp)) dv_1dv_\perp,\ h(v_\perp):=\psi(y_\perp)-\psi(y_\perp-v_\perp).
\end{split}
\ee
Let $\tilde W(y,\eta)$ be the amplitude of $\CW_1\in S^0(\vs)$ in \eqref{B-split}. Then
\be\begin{split}\label{I-stt}
I=&c \int_{|v_\perp|\le c_1P}\int_{-c_1\e}^{c_1\e}\int_\Omega \tilde W(y,\eta) e^{-i(\eta_1 (v_1+h(v_\perp))+\eta_\perp v_\perp)}d\eta dv_1dv_\perp\\
=&c \int_\Omega \frac{\sin(c_1\e\eta_1)}{\eta_1}\int_{|v_\perp|\le c_1 P}\tilde W(y,\eta) e^{-i(\eta_1 h(v_\perp)+\eta_\perp v_\perp)}dv_\perp d\eta.
\end{split}
\ee
Our goal is to show that $I$ is uniformly bounded for all $\e>0$ sufficiently small and $P$ that satisfy $P/\e \ge \varkappa_2>0$. We can select $\CW_{1,2}$ in \eqref{B-split} so that the conic supports of their amplitudes are contained in that of $\CB$.
First, consider only the principal symbol of $\CW_1$, which we denote $\tilde W_0(y,\eta)$. 
We can assume that $\tilde W_0(y,\eta)\equiv0$ if $\eta\not\in\Omega$, $\eta\not=0$, where $\Omega\subset\br^n\setminus\{0\}$ is a small conic neighborhood of $\bp\cup(-\bp)$. This set is used in \eqref{I-stt}. The corresponding value of $I$, which is obtained by replacing $\tilde W(y,\eta)$ with $\tilde W_0(y,\eta)$ in \eqref{I-stt}, is denoted $I_0$. 

As $\tilde W_0(y,\eta)$ is positively homogeneous of degree zero in $\eta$, set
\be\label{Wpm}
\tilde W^\pm(y,u):=\tilde W(y,\eta_1(1,u))=\tilde W_0(y,\pm(1,u)),\ u=\eta_\perp/\eta_1\in\Omega_\perp,
\ee
where $\Omega_\perp$ is a small neighborhood of the origin in $\br^{n-1}$: $\Omega_\perp:=\{u\in \br^{n-1}: u=\eta_\perp/\eta_1,\eta\in\Omega\}$. The sign $'+'$ is selected if $\eta_1>0$, and $'-'$ - otherwise. By the properties of $\CW$, $\tilde W_{\pm}(y,\cdot)\in\coi(\Omega_\perp)$. Thus, \eqref{I-stt} implies
\be\begin{split}\label{I-three}
I_0=&c \int_{\br}\frac{\sin(c_1\e\eta_1)}{\eta_1}  \int_{|v_\perp|\le c_1P}\int_{\Omega_\perp}\tilde W^\pm(y,u) e^{-i\eta_1 u v_\perp}du e^{-i\eta_1 h(v_\perp)} dv_\perp |\eta_1|^{n-1}d\eta_1\\
=&c \int_{\br}\frac{\sin(c_1\e\eta_1)}{\eta_1}  \int_{|v_\perp|\le c_1P}W^\pm(y,\eta_1v_\perp) e^{-i\eta_1 h(v_\perp)} dv_\perp |\eta_1|^{n-1}d\eta_1\\
=&c \int_{\br}\frac{\sin(c_1\e\eta_1)}{\eta_1}  \int_{|w_\perp|\le c_1P|\eta_1|}W^\pm(y,w_\perp) e^{-i\eta_1 h(w_\perp/\eta_1)} dw_\perp d\eta_1\\
=&c \int_{\br}\frac{\sin(\la)}{\la}  \int_{|w_\perp|\le \frac P\e|\la|}W^\pm(y,w_\perp) \exp\left(-i\la \frac{h(c_1\e w_\perp/\la)}{c_1\e}\right) dw_\perp d\la,
\end{split}
\ee
where $W^\pm(y,w_\perp)$ is the inverse Fourier transform of $\tilde W^\pm(y,u)$ with respect to $u$. Since $P/\e$ is bounded away from zero, $h(0)=0$, and $W^\pm(y,w_\perp)$ is smooth and rapidly decreasing as a function of $w_\perp$, we have by the dominated convergence theorem
\be\begin{split}\label{I-four}
&\int_{|w_\perp|\le \frac P\e |\la|}W^\pm(y,w_\perp) \exp\left(-i\la \frac{h(c_1\e w_\perp/\la)}{c_1\e}\right) dw_\perp\\
& \to \int_{\br^{n-1}}W^\pm(y,w_\perp)e^{-i h'(0)\cdot w_\perp}dw_\perp=\tilde W^\pm(y,-\psi'(y_\perp))=\tilde W_0(y,\pm(1,-\psi'(y_\perp)))
\end{split}
\ee
as $\la\to\pm\infty$, and convergence is uniform with respect to $\e$ and $P$ that satisfy $P/\e\ge \varkappa_2$. As is seen, $\bma 1\\ -\psi'(y_\perp) \ema$ is a vector normal to $\Gamma$ at the point $\bma \psi(y_\perp)\\ y_\perp\ema$.

The remainder term in \eqref{I-four} is bounded by the expression
\be\begin{split}\label{I-four-bnd}
&\int_{|w_\perp|\le \frac P\e |\la|}|W^\pm(y,w_\perp)| \left|\exp\left(-i \frac{\la h(c_1\e w_\perp/\la)}{c_1\e}+ih'(0)\cdot w_\perp\right)-1\right| dw_\perp\\
&\qquad+\int_{|w_\perp|\ge \frac P\e |\la|}|W^\pm(y,w_\perp)| dw_\perp
\\
&\le c \frac{\e }{|\la|}\int_{\br^{n-1}}|W^\pm(y,w_\perp)| |w_\perp|^2 dw_\perp
+\int_{|w_\perp|\ge \varkappa_2 |\la|}|W^\pm(y,w_\perp)| dw_\perp=O(|\la|^{-1}).
\end{split}
\ee
Due to $\tilde W_{\pm}(y,\cdot)\in\coi(\Omega_\perp)$, the big-$O$ term on the right-hand side of \eqref{I-four-bnd} is uniform with respect to $y\in\vs$ and $0<\e\le1$. Hence 
\be\begin{split}\label{I-five}
I_0= &c \int_{\br}\frac{\sin(\la)}{\la}  \tilde W_0(y,\la(1,-\psi'(y_\perp))) d\la+O(1)\\
=&c \frac{\pi}2 \left[\tilde W_0(y,(1,-\psi'(y_\perp)))+\tilde W_0(y,-(1,-\psi'(y_\perp)))\right]+O(1),
\end{split}
\ee
where $O(1)$ is uniform with respect to $y\in\vs$ as well, which proves that $I_0$ is uniformly bounded. 

The remaining term $\Delta I=I-I_0$ comes from the subprincipal terms of the amplitude $\Delta \tilde W=\tilde W-\tilde W_0$. The corresponding $\Psi$DO is in $S^{-\nu}(\vs)$ for some $\nu>0$, so its Schwartz kernel $\Delta K(y,w)$ is smooth as long as $w\not=0$ and absolutely integrable at $w=0$. It is now obvious that   $\Delta I$ is bounded as well.

By Lemma \ref{lem:delta} (use \eqref{F-ineq-1} with $l=k=\bt_0$), $|\Delta g_\e^{(\bt_0)}(y)|\le c\e P^{s_0-1-\bt_0}$ if $P\ge \varkappa_1\e$, combining with \eqref{Je1b-parts} proves that $|J_\e^{(1b2)}(y)|\le c\e P^{s_0-1-\bt_0}$. By \eqref{Je1b-parts}, \eqref{J1b1est}, \eqref{J1bdest-st2}, we conclude $|J_\e^{(1b)}(y)|\le c\e P^{s_0-1-\bt_0}\ln(P/\e)$. Combining with \eqref{J-e1-parts} and \eqref{Je1a} we finish the proof.


\section{Some computations involving determinants}\label{sec:dets}

\subsection{Proof of Lemma~\ref{lem:detQ}}
We begin by proving \eqref{x1y1-ders}. 
Differentiating (cf. \eqref{X-ids}):
\be\label{x1pi-repeat}
X_1(\Phi^\tp(t,y),y)\equiv \Phi_1(t,y) 
\ee
with respect to $t$ and using that $\det \Phi^{\tp}_t\not=0$, $\pa\Phi_1/\pa t=0$ gives $\pa X_1/\pa\xt=0$. Differentiating \eqref{x1pi-repeat} with respect to $y_1$ gives $\pa X_1/\pa y_1=\pa \Phi_1/\pa y_1$. Also, from \eqref{y-coords} and \eqref{xcoords}, $|d\Psi|\pa \Phi_1/\pa y_1=1$, i.e. $\pa \Phi_1/\pa y_1>0$. 

We find $y=Y(\yt,x)$ by solving $x=\Phi(t,y)$ for $t$ and $\yo$. Since $\pa\Phi_1/\pa t=0$ and $\pa\Phi_1/\pa y_\perp=0$, differentiating $x_1\equiv\Phi_1(t(\yt,x),(Y^\os(\yt,x),\yt))$ with respect to $x_1$ gives $1=(\pa\Phi_1/\pa y_1)(\pa Y_1/\pa x_1$). Differentiating the same identity with respect to $x_\perp$ gives $0=({\pa \Phi_1}/{\pa y_1})({\pa Y_1}/{\pa x_\perp})$, and all the statements in \eqref{x1y1-ders} are proven.

By \eqref{YZ-diff},
\be\label{Q-compute}
|\det Q|^{1/2}=\frac{|\det C|}{|\det (\Psi\circ\Phi)_{tt}|^{1/2}}.
\ee
Further, by \eqref{xcoords}, \eqref{Bolker-matr}, \eqref{M-alt} and \eqref{M-deriv},
\be\label{Mdets}\begin{split}
|\det C|=\frac{|\det M|}{|\det M_{11}|}
=|d\Psi|^N\left|\det \Phi^\tp_t\right|^{-1}{\left|\det \bma \Phi^\tp_t & \Phi^\tp_{\yt} \\
(\Phi_1)_{tt} & (\Phi_1)_{t\yt} \ema\right|}.
\end{split}
\ee
Differentiating \eqref{x1pi-repeat} two times 
gives
\be\label{phi1tt}
(\Phi_1)_{tt}=L\Phi^\tp_t,\quad L:=(\Phi^\tp_t)^T (X_1)_{\xt\xt}.
\ee
Here we have used that $\pa X_1/\pa \xt=0$. 
Similarly,
\be\label{phi1ty}
(\Phi_1)_{ty}=L\Phi^\tp_{\yt}+(\Phi^\tp_t)^T (X_1)_{\xt\yt}.
\ee
This simplifies the last determinant in \eqref{Mdets}:
\be\label{last-det}\begin{split}
\det \bma \Phi^\tp_t & \Phi^\tp_{\yt} \\
(\Phi_1)_{tt} & (\Phi_1)_{t\yt} \ema
=\det \Phi_t^{\tp}\det \bma I & \Phi^\tp_{\yt} \\
L & L\Phi^\tp_{\yt}+(\Phi^\tp_t)^T (X_1)_{\xt\yt} \ema\\
=\det \Phi_t^{\tp}\det \bma I & \Phi^\tp_{\yt} \\
0 & (\Phi^\tp_t)^T (X_1)_{\xt\yt} \ema
=(\det \Phi_t^{\tp})^2\det\frac{\pa^2 X_1}{\pa \xt\pa\yt},
\end{split}
\ee
so 
\be\label{detC}\begin{split}
|\det C|
=\left|\frac{\pa X_1}{\pa y_1}\right|^{-N}|\det \Phi_t^{\tp}|\left|\det\frac{\pa^2 X_1}{\pa \xt\pa\yt}\right|.
\end{split}
\ee
Combining everything and using \eqref{jacob-sm} gives \eqref{Q-final}.


\subsection{Proof of Lemma~\ref{lem:chi}}
Transform the combination of the determinants in \eqref{dets-only-v2} (cf. \eqref{YZ-diff})
\be\label{dets-only}\begin{split}
\chi=\frac{(\det G^\s )^{1/2}}{|\det C|}.
\end{split}
\ee
By \eqref{xcoords} and \eqref{gram-1st}, 
\be\label{st1}
G_{jk}^\s=\frac{\pa \Phi^{\tp}(t,y)}{\pa t_j}\cdot\frac{\pa \Phi^{\tp}(t,y)}{\pa t_k},
\ee
and $(\det G^\s)^{1/2}=|\det \Phi^\tp_t|$. Using \eqref{detC} establishes \eqref{dets-only-v2}.

\bibliographystyle{plain}
\bibliography{bibliogr_A-K,bibliogr_L-Z}
\end{document}